\documentclass{amsart}
\usepackage[nohug]{diagrams}
\usepackage[mathscr]{eucal}
\usepackage{amsfonts}
\usepackage{amsmath}
\usepackage{amsthm}
\usepackage{amssymb}
\usepackage{latexsym}

\usepackage[noadjust]{cite}

\usepackage{graphicx}
\usepackage{color}
\usepackage{epsfig}

\newfont{\msam}{msam10}
\newtheorem{theorem}{Theorem}[section]

\newtheorem{proposition}[theorem]{Proposition}
\newtheorem{corollary}[theorem]{Corollary}
\newtheorem{lemma}[theorem]{Lemma}

\theoremstyle{definition}
\newtheorem{definition}[theorem]{Definition}
\newtheorem{remark}[theorem]{Remark}

\newtheorem{example}[theorem]{Example}
\newtheorem{conjecture}[theorem]{Conjecture}

\let\nc\newcommand

\nc{\la}{\label}

\def\bthm{\begin{theorem}}
\def\ethm{\end{theorem}}
\def\blemma{\begin{lemma}}
\def\elemma{\end{lemma}}
\def\bproof{\begin{proof}}
\def\eproof{\end{proof}}
\def\bprop{\begin{proposition}}
\def\eprop{\end{proposition}}

\def\b{\mathfrak b}

\def\Z{\mathbb{Z}}

\def\O{\mathcal{O}}

\def\H{\mathscr{H}}

\def\e{\boldsymbol{\mathrm{e}}}

\def\ult{{\underline{t}}}
\def\ulone{{\underline{1}}}
\def\loc{\mathrm{loc}}

\def\yy{\hat{y}}
\def\xx{\hat{x}}

\def\c{\mathbb{C}}
\def\C{\mathbb{C}}

\nc{\Hom}{{\rm{Hom}}}
\nc{\Ext}{{\rm{Ext}}}
\nc{\htau}{{\bar{ t}}}
\nc{\HOM}{\underline{\rm{Hom}}}
\nc{\EXT}{\underline{\rm{Ext}}}
\nc{\TOR}{\underline{\rm{Tor}}}
\nc{\End}{{\rm{End}}}
\nc{\Map}{{\rm{Map}}}
\nc{\Out}{{\rm{Out}}}
\nc{\GL}{{\rm{GL}}}
\nc{\SL}{{\rm{SL}}}
\nc{\PGL}{{\rm{PGL}}}
\nc{\G}{{\rm{G}}}
\nc{\Rep}{{\rm{Rep}}}
\nc{\ad}{{\rm{ad}}}
\nc{\dlim}{\varinjlim}

\def\H{\mathrm{H}}
\def\SH{\mathrm{SH}}

\def\S{\mathrm{S}}

\newcommand{\Aut}{{\rm{Aut}}}

\newcommand{\Tr}{{\rm{Tr}}}
\newcommand{\tr}{{\rm{tr}}}
\newcommand{\Ker}{{\rm{Ker}}}

\newcommand{\ev}{{\rm{ev}}}

\newcommand{\onto}{\,\,\twoheadrightarrow\,\,}

\numberwithin{equation}{section}

\nc{\sk}{{\mathrm{Sk}}}

\newcommand{\ochar}{{\mathcal O \mathrm{Char}}}
\newcommand{\chr}{{\mathrm{Char}}}

\newcommand{\rep}{ {\mathrm{Rep} } }

\newcommand{\pmone}{ {\pm 1 } }
\newcommand{\pp}{\mathfrak{p}}

\newcommand{\lv}[1]{ {\lvert #1 \rvert } }

\newcommand{\rr}[1]{\textcolor{red}{#1}}
\newcommand{\cmt}[1]{\noindent \rr{\textbf{#1}}}

\newcommand\bto{\bar {t}_1}
\newcommand{\btt}{\bar {t}_2}

\begin{document}
\title{Affine cubic surfaces and character varieties of knots}
\date{\today}
\author{Yuri Berest}
\address{Department of Mathematics, Cornell University, Ithaca, NY 14853-4201, USA}
\email{berest@math.cornell.edu}
%
%
\author{Peter Samuelson}
\address{Department of Mathematics, University of Edinburgh, Edinburgh, UK, EH9 1PH}
\email{peter.samuelson@ed.ac.uk}

\begin{abstract}
It is known that the fundamental group homomorphism $\pi_1(T^2) \to \pi_1(S^3\setminus K)$  induced by the inclusion
of the boundary torus into the complement of a knot $K$ in $S^3$ is a complete knot invariant. Many classical  invariants
of knots arise from the natural (restriction) map induced by the above homomorphism on the $\SL_2$-character varieties of the corresponding fundamental groups. In our earlier work \cite{BS16}, we proposed a
conjecture that the classical restriction map admits a canonical
2-parameter deformation into a smooth cubic surface. In this paper,
we show that (modulo some mild technical conditions) our conjecture follows from a known conjecture of Brumfiel and Hilden \cite{BH95}
on the algebraic structure of the peripheral system of a knot. We then confirm the Brumfiel-Hilden conjecture for an infinite class of knots, including all torus knots, 2-bridge knots, and certain pretzel knots. We also show the class of knots for which the Brumfiel-Hilden conjecture holds is closed under taking connect sums and certain knot coverings.
\end{abstract}

\maketitle
\begin{center}
\textit{To Efim Zelmanov on the occasion of his 60th birthday}
\end{center}
\setcounter{tocdepth}{1}
\tableofcontents

\section{Introduction and motivation}


One classical tool in the study of 3-manifolds is the $\SL_2(\C)$ 
\emph{character variety} $\chr(M)$. This is the (algebro-geometric) quotient of the representation variety $\Rep(M)$ by the natural $\GL_2(\C)$ action, where $\Rep(M)$ parameterizes representations $\pi_1(M) \to \SL_2(\C)$ of the fundamental group of a 3-manifold $M$ into $\SL_2(\C)$, and $\GL_2(\C)$ acts by conjugation.  

In particular, given a knot $K \subset S^3$, there is a natural map $\alpha: \chr(S^3\setminus K) \to \chr(T^2)$ given by restricting representations of a knot complement to its boundary. The image of this map determines and is (essentially) determined by the $A$-polynomial of $K$ (see \cite{CCG94}), and this polynomial contains a good deal of geometric and topological information about the knot complement. For example, one of the main results of \cite{CCG94} asserts that slopes of the boundary of the Newton polygon of the $A$-polynomial determine boundary slopes of incompressible surfaces of $S^3\setminus K$. Additionally, in \cite{DG04} Dunfield and Garoufalidis used work of Kronheimer and Mrowka \cite{KM04} to show that the $A$-polynomial distinguishes the unknot.

In this paper we study 
a 2-parameter family $\alpha_{t_1,t_2}$ of deformations of the restriction map $\alpha$ from the character variety $\chr(S^3\setminus K)$ to certain affine cubic surfaces in $\C^3$:
\begin{equation}\label{eq_familyintro}
\alpha_{t_1,t_2}: \chr(S^3\setminus K) \to X_{t_1,t_2}
\end{equation}
The special fiber $X_{1,1}$ of this family is isomorphic to the character variety of the torus $T^2$, and the specialization $\alpha_{1,1}$ reproduces the classical restriction map
$ \alpha: \chr(S^3\setminus K) \to X_{1,1}$. 

The $q=-1$ specialization of the main conjecture of \cite{BS16} states that there is a \emph{canonical} deformation of $\alpha$ to the map $\alpha_{t_1,t_2}$. For general $q$, this conjecture involves a quantization of the character variety of the knot complement. It seems generally agreed (at the moment) that the quantization (or $q$-deformation) of character varieties of knots requires the use of topological tools, such as the Kaufmann bracket skein module construction which was used in \cite{BS16}. By contrast, the Hecke (or Dunkl) deformations that we study in the present paper (for $q=-1$) depend only on the knot group and
may be performed purely algebraically (using the Brumfiel-Hilden algebra). This is, perhaps, the main observation of the present paper. (See Figure \ref{fig_objects}.)

Below, we will briefly describe the origin of this conjecture, its relation to the character variety of the 4-punctured sphere, and an interpretation in terms of the Brumfiel-Hilden algebra. 
We then describe our results which confirm this conjecture for an infinite family of knots, including torus knots, 2-bridge knots, some pretzel knots, and all connect sums of these.

\begin{figure}
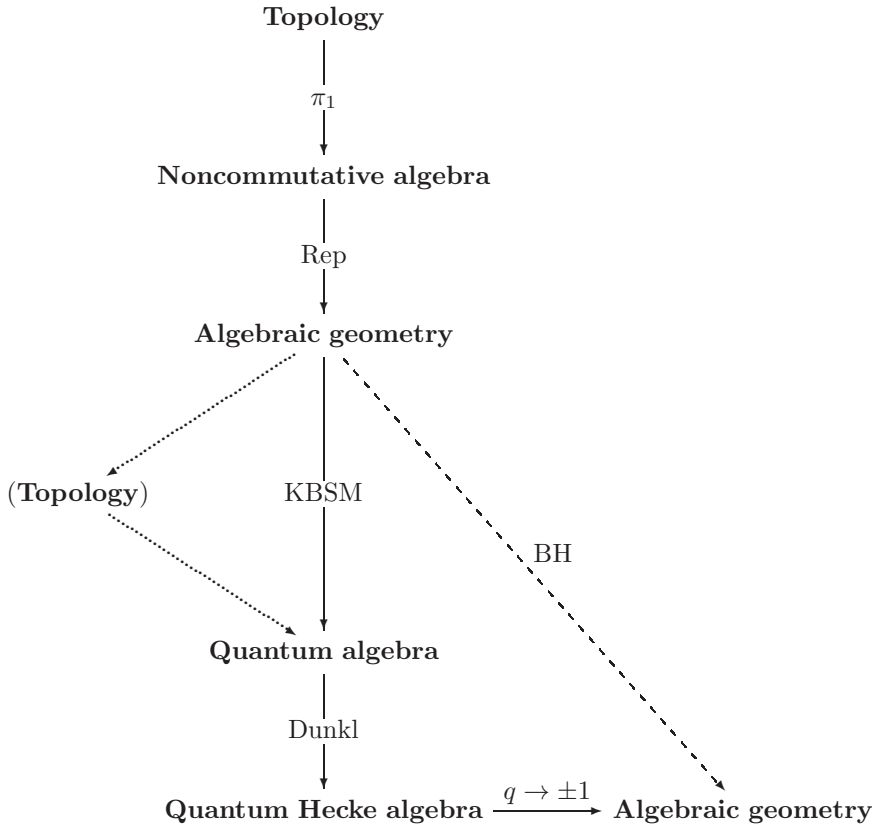

$$
\begin{diagram}
& & & {\bf Topology}  & & &\\
& & &  \dTo~{\pi_1}  & & & \\
& & & {\bf Noncommutative\ algebra}  & & &\\
&  &  & \dTo~{\rm Rep} & & & \\
& &  &  {\bf Algebraic\ geometry} & & &\\
& & \ldDotsto   &    &\quad  \rdDashto(2,6)^{\rm BH} & & \\
&({\bf Topology}) & & \dTo~{\rm KBSM} & & &\\
& & \rdDotsto   &    &  & &\\
& & &  {\bf Quantum\ algebra}  & & & \\
& & & \dTo~{\rm Dunkl}  & & & \\
& &  & {\bf Quantum\ Hecke\ algebra}  & \rTo^{q \to \pm 1\ }  & {\bf Algebraic\ geometry}
\end{diagram}
$$
\caption{Deformations vs. quantizations of character varieties of knots}\label{fig_objects}
\end{figure}

The \emph{Kauffman bracket skein module} $\sk_q(M)$ of a 3-manifold $M$ is the $\C[q^{\pm 1}]$-module spanned by framed, unoriented links in $M$ modulo the Kauffman bracket skein relations (see Figure \ref{fig_kbsm}). If $M = F \times I$ is a thickened surface, then $\sk_q(F\times I)$ is an algebra, where the multiplication is given by stacking one link on top of another. Similarly, the space $\sk_q(M)$ is a module over the algebra $\sk_q((\partial M)\times I)$ associated to the boundary of $M$. 

The (spherical) double affine Hecke algebra (DAHA) $\SH_{q,\ult}$ is a noncommutative algebra depending on a parameter $q$ and four additional parameters $\ult = (t_1,t_2,t_3,t_4)$. It follows easily from a theorem of Frohman and Gelca \cite{FG00} that the specialization $\SH_{q,\ulone}$ is isomorphic to the skein algebra $\sk_q(T^2)$ of the torus. This implies that the skein module $\sk_q(S^3\setminus K)$ of a knot complement is a module over $\SH_{q,1,1,1,1}$. Based on explicit computations in some examples, the following conjecture was proposed in \cite{BS16}.
\begin{conjecture}[{\cite{BS16}}]\label{conj_bs14}
 The double affine Hecke algebra $\SH_{q,t_1,t_2,1,1}$ acts canonically on the Kauffman bracket skein module $\sk_q(S^3\setminus K)$ of the complement of a knot $K \subset S^3$.
\end{conjecture}

The connection between this conjecture and character varieties follows from a theorem of Bullock \cite{Bul97} (see also \cite{PS00}), which states that the $q=-1$ specialization of the skein module $\sk_{q=-1}(M)$ is a commutative ring which is isomorphic to $\O(\chr(M))$. It turns out that in the $q=\pm 1$ specialization the DAHA $\SH_{q=\pm 1,\ult}$ is commutative, and it was studied thoroughly by Oblomkov in \cite{Obl04}. It follows from \cite{Obl04} and work of Goldman in \cite{Gol97} that $\SH_{1,\ult}$ is isomorphic to the ring of functions $\O\chr(S^2\setminus \{p_1,p_2,p_3,p_4\})$ on the character variety of the 4-punctured sphere. (In fact, it follows from work of Bullock and Przytycki \cite{BP00} and Terwilliger \cite{Ter13} that $\SH_{q^2,\ult}$ is isomorphic to the skein algebra $\sk_q(S^2\setminus \{p_1,p_2,p_3,p_4\})$.)

We now describe the construction of the family \eqref{eq_familyintro} of maps that deforms the peripheral map. One of the key properties of the DAHA is the so-called `Dunkl embedding,' which is an injective algebra homomorphism $\SH_{q,\ult} \hookrightarrow \SH_{q,\ulone}^\loc$ into a localization of the DAHA at $t_i = 1$. When $q=-1$, this becomes a rational map $X_{\ulone} \dashrightarrow X_{\ult}$. In this language, Conjecture \ref{conj_bs14} becomes: 
\begin{conjecture}(Conjecture \ref{conj_bs14} at $q=-1$): The composition of the restriction map $\chr(S^3\setminus K) \to X_{\ult = \ulone}$ with the Dunkl embedding $X_{\ult = \ulone} \dashrightarrow X_{t_1,t_2,1,1}$ extends to a regular morphism of affine schemes. 
\end{conjecture}
We remark that this is not automatic because the poles of the rational map $X_\ulone \to X_\ult$ contain the trivial representation, and the poles therefore intersect the image of the map $\chr(S^3\setminus K) \to X_\ulone$ for any knot $K$.
In this paper we confirm this conjecture at $q=-1$ for an infinite class of knots (at least up to a technical condition, see Corollary \ref{cor_BH_def}). 

A useful tool for studying character varieties of a discrete group $\pi$ is the \emph{Brumfiel-Hilden algebra}, which we denote $H[\pi]$. This algebra (and its `trace' subalgebra) are defined by
\[
H[\pi] := \frac{\C[\pi]}{\{h(g+g^{-1}) = (g+g^{-1})h\} },\quad \quad H^+[\pi] := \{a \in H[\pi] \mid a = \sigma(a)\}
\]
where $\sigma: H[\pi] \to H[\pi]$ is the anti-automorphism defined on group elements via $\sigma(a) = a^{-1}$. The key theorem (see \cite[Prop.~9.1]{BH95}) is that $\O\chr(\pi) \cong H^+[\pi]$.

We show that Conjecture \ref{conj_bs14} (at $q=-1$) has a  natural interpretation in terms of $H[\pi]$, where $\pi$ is the fundamental group $\pi_1(S^3\setminus K)$ of the knot complement. We will call the condition in the following conjecture the \emph{Brumfiel-Hilden condition}.

\begin{conjecture}[{\cite{BH95}}] Let $M$ and $L$ be the standard meridian and longitude of the knot. Then
\begin{equation}\label{eq_bhintro}
L \in H^+[M^{\pm 1}]
\end{equation}
where $H^+[M^{\pm 1}]$ is the subalgebra of $H[\pi]$ generated by $H^+[\pi]$ and $M^{\pm 1}$. 
\end{conjecture}

We prove the following theorem (see Corollary \ref{cor_BH_def}):
\begin{theorem}\label{thm_bhusintro}
If the Brumfiel-Hilden condition \eqref{eq_bhintro} holds and 
$(M-M^{-1}):H^+[\pi] \to H^+[\pi][M^{\pm 1}]$ 
is injective, then Conjecture \ref{conj_bs14} holds at $q=-1$.
\end{theorem}

In further support of both these conjectures, we prove the following.

\begin{theorem}\label{thm_mainintro}
The Brumfiel-Hilden condition \eqref{eq_bhintro} holds for all torus knots, 2-bridge knots, and certain $(-2,3,2n+1)$ pretzel knots. Furthermore, if it holds for knots $K$ and $K'$, then it holds for their connect sum $K\# K'$.
\end{theorem}

The contents of the paper are as follows. In Section \ref{sec_daha_charvarieties} we recall background information about character varieties and double affine Hecke algebras. In Section \ref{sec_bh} we recall the Brumfiel-Hilden algebras and prove Theorem \ref{thm_bhusintro}. In Section \ref{sec_connectsum} we show that the Brumfiel-Hilden condition is preserved by connect sum of knots and by certain coverings of knots. The proof of Theorem \ref{thm_mainintro} for torus knots is contained in Section \ref{sec_torus}, and the proof for certain pretzel knots is in Section \ref{sec_pretzel}. We confirm the Brumfiel-Hilden condition for 2-bridge knots in Section \ref{sec_twobridge}. Further remarks are contained in Section \ref{sec_furtherremarks}, and the Appendix contains an alternative description of the Brumfiel-Hilden algebra for 2-generator 1-relator groups.

\textbf{Acknowledgements:} We are thankful to P.~Boalch, F.~Bonahon, O.~Chalykh, C.~Dunkl,    D.~Muthiah, V.~Roubtsov, S.~Sahi, and P.~Terwilliger for helpful discussions regarding their work and/or the present paper. 
The first author (Y.~B.) would like to thank the scientific
committee of the international conference ``Lie algebras and Jordan algebras, their applications and representations (dedicated to Efim Zelmanov 60th birthday)'' for inviting him to give a plenary talk.
He is very grateful to the local organizers of the conference,
especially V. Futorny and I. Kashuba, for their warm hospitality and support in Brazil. The work of Peter Samuelson was funded in part by European Research Council grant no. 637618.

\section{Double affine Hecke algebras and character varieties of surfaces} \label{sec_daha_charvarieties}
In this section we describe the relationship between the $C^\vee C_1$ (spherical) double affine Hecke algebra $\S\H_{q,\ult}$ and the Kauffman bracket skein algebra $\sk_q(S^2\setminus\{p_1,p_2,p_3,p_4\})$ of the 4-punctured sphere. 
This implies a relationship between the $q=1$ specialization of the (spherical) DAHA and the relative $\SL_2(\C)$ character varieties of the 4-punctured sphere, which we describe explicitly. Finally, the polynomial representation of the DAHA gives an embedding of $\S\H_{q,\ult}$ into a localization of the skein algebra $ \sk_q(T^2)$, which we also describe explicitly. This gives explicit formulas for the rational map $\chr(T^2) \dashrightarrow \chr(S^2\setminus\{p_i\})$ which we provide in Corollary \ref{cor_symdunkl}. We conclude with explicit formulas describing the family \eqref{eq_familyintro} for the trefoil and figure eight knots.


\subsection{Character varieties of topological surfaces and affine cubic surfaces}
In this section we recall some results of Goldman in \cite[Sec.\ 6]{Gol97}. Let $\pi$ be the fundamental group of a 4-punctured sphere, with a presentation 
\[
\pi = \langle A,B,C,D \mid ABCD = \mathrm{Id}\rangle
\]
where each generator corresponds to a loop around a puncture. Consider the following seven functions on the $\SL_2$ character variety of $\pi$:
\begin{equation*}
\begin{array}{c}
a = \tr(A),\quad b = \tr(B),\quad c = \tr(C),\quad d = \tr(D) \\
x_S = \tr(AB),\quad y_S = \tr(BC),\quad z_S = \tr(CA)\\
\Omega_C := x_S^2+y_S^2+z_S^2+x_Sy_Sz_S - (ab+cd)x_S - (ad+bc)y_S - (ac+bd)z_S
\end{array}
\end{equation*}
These functions satisfy the defining equation
\begin{equation}
\Omega_C =  - (a^2+b^2+c^2+d^2+abcd) + 4\label{eq_goldman}
\end{equation}

\begin{theorem}[\cite{Gol97}]\label{thm_goldmanrep}
The relation \eqref{eq_goldman} describes an embedding
\[
\chr (S^2\setminus \{p_1,p_2,p_3,p_4\}) \hookrightarrow \C^7 
\]
\end{theorem}

\begin{remark}
There is a map $\C^7 \to \C^4$ given by the coordinates $a$, $b$, $c$, and $d$, and the \emph{relative character variety} is a fiber of this map. These fibers can be viewed as cubic surfaces in $\C^3$, and they first appeared in the work of Vogt \cite{Vog89} and Fricke and Klein \cite{FK65} on invariant theory in the late 19th century (see also \cite{Mag80}). In recent years they have found many interesting applications: for example, as monodromy surfaces of the classical Painleve VI equation (see, e.g. \cite{Iwa03} and \cite{IIS06}).
\end{remark}

\subsection{The Kauffman bracket skein algebra}
Here we give some very brief background about the Kauffman bracket skein module $\sk_q(M)$ of a 3-manifold, and refer to other works for more details (e.g. \cite{BS16} and references therein). Given an oriented manifold $M$, the skein module $\sk_q(M)$ is the vector space formally spanned by framed links in $M$ modulo the Kauffman bracket skein relations. 

\begin{figure}
\begin{center}
\input{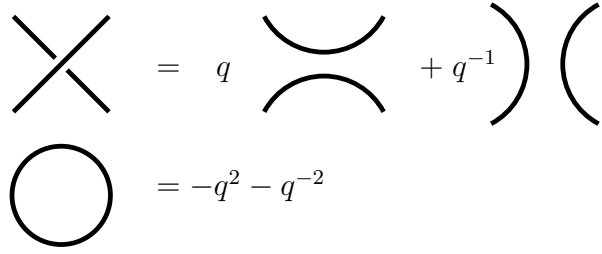}
\caption{Kauffman bracket skein relations}\label{fig_kbsm}
\end{center}
\end{figure}

If $M = F \times [0,1]$ is a thickened surface, then $\sk_q(F \times [0,1])$ is an algebra, where the multiplication is given by stacking in the $[0,1]$ direction. Also, for any 3-manifold $M$, if $q=\pm 1$, then $\sk_{q=\pm 1}(M)$ is a commutative algebra, where the product is given by disjoint union (this product is only defined at the specializations $q=\pm 1$.). This commutative algebra is related to character varieties via the following theorem.

\begin{theorem}[\cite{PS00}, \cite{Bul97}]\label{thm_orep}
 The map $\gamma \mapsto -\Tr_\gamma$ extends to an isomorphism of commutative algebras
 \[
\sk_{q=-1}(M) \stackrel \sim \longrightarrow \O\chr(M)  
 \]
where $\gamma$ is a loop and $\Tr_\gamma(\rho) := \Tr(\rho(\gamma))$.
\end{theorem}

We also use a presentation of the skein algebra $\sk_q(S^2\setminus \{p_i\})$ of the 4-punctured sphere given by Bullock and Przytycki. Let $x_1$ and $x_2$ be two distinct simple closed curves in $S^2 \setminus \{p_i\}$ which are non-boundary parallel and which intersect twice.(See Figure \ref{fig_curves}.) Define the curve $x_3$ via the equation
\[
 x_1x_2 = q^2 x_3 + q^{-2} z + \textrm{boundary curves}
\]
where $x_3$ and $z$ are simple closed curves each of which intersect $x_1$ and $x_2$ in two points. Suppose $x_1$ separates boundary curves $a_1$ and $a_2$ from $a_3$ and $a_4$, and define $p_1 = a_1a_2 + a_3a_4$. Define $p_2$ and $p_3$ similarly. Finally, define 
\[
 \Omega_K := -q^2 x_1x_2x_3 + q^4 x_1^2+q^{-4}x_2^2 + q^4x_3^2 + q^2 p_1 x_1 + q^{-2}p_2 x_2 + q^2 p_3 x_3
\]

\begin{figure}
\begin{center}
\input{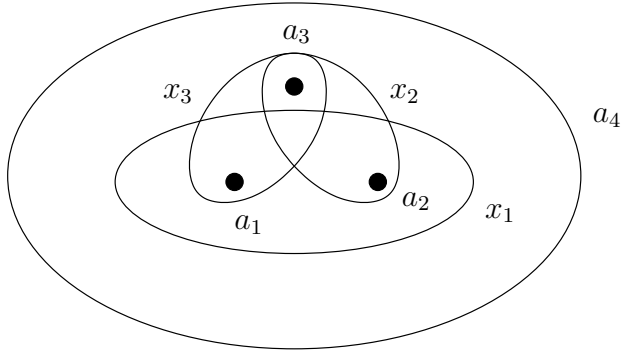}
\caption{Curves on the $4$-punctured sphere}
\label{fig_curves}
\end{center}
\end{figure}

\begin{theorem}[{\cite[Thm.\ 3]{BP00}}]\label{thm_ksphere}
With notation as in the previous paragraph, $\sk_q(S^2\setminus \{p_i\})$ has a presentation where the generators are $x_i$ and $a_i$ and the relations are
\begin{eqnarray*}
 [x_i,x_{i+1}]_{q^2} &=& (q^4-q^{-4})x_{i+2} - (q^2-q^{-2})p_{i+2}\\
 \Omega_K &=&  (q^2+q^{-2})^2 -( a_1a_2a_3a_4 + a_1^2+a_2^2+a_3^2+a_4^2)
\end{eqnarray*}
(The indices in the first relation are interpreted modulo 3.)
\end{theorem}

\begin{remark}
 It is clear from the formulas above that Theorems \ref{thm_orep} and \ref{thm_ksphere} are compatible with Theorem \ref{thm_goldmanrep}, where $(x,y,z)$ correspond to $(-x_1,-x_2,-x_3)$ and $(a,b,c,d)$ correspond to $(-a_1,-a_2,-a_3,-a_4)$.
\end{remark}

\subsection{The $C^\vee C_1$ double affine Hecke algebra}
In this section we recall the 5-parameter family of algebras $\H_{q,\ult}$ which was introduced by Sahi in \cite{Sah99} (see also \cite{NS04}).
This is the universal deformation of the algebra $\C[X^{\pm 1}, Y^{\pm 1}] \rtimes \Z_2$ (see \cite{Obl04}), and it depends on the parameters $q \in \C^*$ and $\ult \in (\C^*)^4$. The algebra $\H_{q,\ult}$ can be abstractly presented as follows: it is generated by the elements $T_1$, $T_2$, $T_3$, and $T_4$ subject to the relations
\begin{align}\label{ccdaharelations}
 (T_i-t_i)(T_i+t_i^{-1}) &= 0,\quad \quad 1 \leq i \leq 4\\
 T_4 T_3T_2 T_1 &= q\notag
\end{align}
\begin{remark}
Comparing our notation to \cite{BS16}, their $(T_0,T_0^\vee,T_1,T_1^{\vee})$ are our $(T_2,T_1,T_3,T_4)$, and their $(t_1,t_2,t_3,t_4)$ are our $(t_2,t_1,t_3,t_4)$.
\end{remark}

The element $\e := (T_3+t_3^{-1})/(t_3+t_3^{-1})$ is an idempotent in $\H_{q,\ult}$, and the algebra $\S\H_{q,\ult} := \e \H_{q,\ult}\e$ is called the \emph{spherical subalgebra}. 
A presentation for the spherical subalgebra $\S\H_{q,t}$ has been given in \cite{Ter13}, and this can be viewed as a $q$-deformation of the presentation given by Oblomkov in \cite{Obl04}. (A less symmetric presentation was given in \cite{Koo08}. See also \cite{Ter11} and \cite{IT10}.) We now recall this presentation in our notation. Define 
\begin{eqnarray*}
x &=& (T_4T_3 + (T_4T_3)^{-1})\e\\ 
y &=& (T_3T_2+(T_3T_2)^{-1})\e\\
z &=& (T_3T_1 + (T_3T_1)^{-1})\e\\
\Omega_D &=& -qxyz + q^2x^2 + q^{-2} y^2 + q^{2}z^2 - q\alpha x - q^{-1} \beta y - q \gamma z
\end{eqnarray*}
where 
\[
\alpha := \bar t_1 \bar t_2 + (\overline{q t_3})\bar t_4,
\quad \beta := \bar t_1 \bar t_4 + (\overline{q t_3})\bar t_2,
\quad \gamma := \bar t_2 \bar t_4 + (\overline{q t_3})\bar t_1 
 \]
Here and later we use the notation 
\[ 
\bar t_i := t_i - t_i^{-1},\quad \quad \overline{qt_3} := qt_3-q^{-1}t_3^{-1}
\] 
The following theorem is a slight modification of a result of Terwilliger -- for a proof of the modified statement, along with explanations regarding notational conventions, see \cite[Thm.~2.20]{BS16}. 
\begin{theorem}[{\cite[Prop 16.4]{Ter13}}]\label{thm_terwilliger}
The spherical subalgebra $\S\H_{q,\ult}$ is generated by $x,y,z$ with relations
\begin{eqnarray*}
 [x,y]_q &=& (q^2-q^{-2})z - (q-q^{-1})\gamma \\
{ } [y,z]_q &=& (q^2-q^{-2})x - (q-q^{-1})\alpha \\
{ } [z,x]_q &=& (q^2-q^{-2})y - (q-q^{-1})\beta \\
 \Omega_D &=& (\bar t_1)^2 + (\bar t_2)^2 + (\overline{q t_3})^2 + (\bar t_4)^2 - \bar t_1 \bar t_2 (\overline{q t_3}) \bar t_4 + (q+q^{-1})^2
\end{eqnarray*}
\end{theorem}
\begin{remark}
Here we have corrected a typo from \cite{BS16} in the powers of $q$ in the last term of the relation involving $\Omega_D$. We have also slightly rewritten $\Omega_D$ using the commutation relations above.
\end{remark}

\begin{remark}
 If $q=\pm 1$ then the spherical subalgebras are commutative for any $(t_1,t_2,t_3,t_4)$. The corresponding varieties are affine cubic surfaces studied in detail in \cite{Obl04} (and the presentation in \cite{Obl04} agrees exactly with the one above, where our $x,y,z$ are his $X_1,X_2,X_3$). 
\end{remark}

Using these explicit presentations, we now relate the skein module of the 4-punctured sphere with the spherical DAHA. We point out that we must replace $q$ by $q^2$ in the DAHA to define this map.
\begin{corollary}
 Let $i^2=-1$. There is an algebra map $\sk_q(S^2\setminus \{p_1,p_2,p_3,p_4\}) \to \S\H_{q^2,\ult}$ given by
\begin{equation*}
 \begin{array}{c}
  x_1 \mapsto x,\quad x_2 \mapsto y,\quad x_3 \mapsto z\\
  a_1 \mapsto i\bar t_1,\quad a_2 \mapsto i \bar t_2,\quad a_3 \mapsto i(\overline{qt_3}),\quad a_4 \mapsto i\bar t_4
 \end{array}
\end{equation*}
\end{corollary}
\begin{remark}
 The appearance of $\sqrt{-1}$ here has a heuristic explanation as follows. The standard relation for Hecke algebras (with braid generator $T$ and parameter $t$) is given by $(T-t)(T+t^{-1}) = 0$. This can be rewritten as $T-T^{-1} = t-t^{-1}$. Given a matrix $A \in \SL_2(\C)$ with eigenvalues $a$ and $a^{-1}$, the following matrix equation is satisfied: $A + A^{-1} = (a+a^{-1})\mathrm{Id}$. Then the matrix equation can be obtained from the Hecke relation by rescaling $T$ and $t$ by $\sqrt{-1}$.
\end{remark}

We now specialize this corollary to obtain a map $\ochar(S^2\setminus \{p_i\}) \to \S\H_{q=1,\ult}$. We remark that here we specialize $q=1$ because the $q$ in the DAHA is replaced by $q^2$ when it is compared to the skein algebra of the 4-punctured sphere.
\begin{corollary}
 There is a map of commutative algebras $\ochar(S^2\setminus \{p_1,p_2,p_3,p_4\}) \to \S\H_{q=1,\ult}$ 
 \[
  x_S \mapsto -x,\quad y_S \mapsto -y,\quad z_S \mapsto -z,\quad a_i \mapsto \sqrt{-1} \, \bar t_i
 \]
\end{corollary}


\subsection{The unpunctured torus}
The following theorem was proved in \cite{BP00} (for a different but conceptually appealing description of the same algebra, see \cite{FG00}). Let $x_T$, $y_T$, and $z_T$ be the $(1,0)$, $(0,1)$, and $(1,1)$ curves on the torus $T^2$. Let $\Omega_T = -qx_Ty_Tz_T + q^2x_T^2 + q^{-2}y_T^2 + q^2z_T^2$.
\begin{theorem}[\cite{BP00}]
 The algebra $\sk_q(T^2)$ is generated by $x_T$, $y_T$, and $z_T$ subject to the following relations:
 \begin{eqnarray*}
  \,[x_T,y_T]_q &=& (q^2-q^{-2})z_T\\
  \,[z_T,x_T]_q &=& (q^2-q^{-2})y_T\\
  \,[y_T,z_T]_q &=& (q^2-q^{-2})x_T\\
  \,\Omega_T &=& 2(q^2+q^{-2})
 \end{eqnarray*}

\end{theorem}

We can combine this with the previous theorems to obtain the following.
\begin{corollary}\label{cor_toruschar}
 There is an algebra isomorphism\footnote{To be pedantic, if the base ring for $\S\H_{q,\ult}$ is $\C[q^{\pm 1},\ult^{\pm 1}]$, and if $\C$ is given a $\C[q^{\pm 1},\ult^{\pm 1}]$-module structure where the $t_i$ act by $1$, then $\S\H_{q,\ult}\otimes_{\C[q^{\pm 1},\ult^{\pm 1}]} \C \to \sk_q(T^2)$ is an isomorphism.} $\S\H_{q,1,1,1,1} \to \sk_q(T^2)$ given by 
 \[
  x \mapsto x_T,\quad y \mapsto y_T,\quad z \mapsto z_T,\quad t_j \mapsto 1
 \]
 There is a surjective algebra map $\sk_q(S^2\setminus \{p_i\}) \to \sk_{q^2}(T^2)$ given by 
 \[
  x_1 \mapsto x_T,\quad x_2 \mapsto y_T,\quad x_3 \mapsto z_T,\quad a_1,a_2,a_4\mapsto 0,\quad a_3 \mapsto (iq) + (iq)^{-1}
 \]
 where $i^2=-1$.
\end{corollary}

Using the polynomial representation in the next section, we will extend the first map in the previous corollary to the parameters $\S\H_{q,t_1,t_2,1,1}$, at the expense of expanding the range by localizing at certain elements.

\subsection{The polynomial representation}
The DAHA $\H_{q,\ult}$ can be realized by operators on Laurent polynomials $\C[X^{\pm 1}]$ as follows. First, we define auxiliary operators on $\C[X^{\pm 1}]$:
\[
\xx \cdot f(X) = Xf(X),\quad \quad s\cdot f(X) = f(X^{-1}),\quad \quad \yy\cdot f(X) = f(q^{-2}X)
\]
We then define
\begin{eqnarray*}
 \hat T_2 &=& t_2 s\yy - \frac{q^2 \bar t_2 \xx^2  + q\bar t_1 \xx}{1-q^2\xx^2}(1-s\yy)\\
 \hat T_3 &=& t_3s + \frac{\bar t_3 +\bar t_4\xx}{1-\xx^2}(1-s)
\end{eqnarray*}
The operator $\hat T_2$ acts on Laurent polynomials because $(1-s\yy)\cdot X^n = X^n - q^{-2n}X^{-n}$ is divisible by $1-q^2X^2$ (and similarly for $\hat T_3$).
The following Dunkl-type embedding is defined using these operators (see \cite[Thm. 2.22]{NS04}):
\begin{proposition}[\cite{Sah99}]\label{prop_dunklembedding}
 The assignments 
 \begin{equation}\label{ccdunklembedding}
 T_1 \mapsto q\hat T_2^{-1}\xx,\quad T_2 \mapsto \hat T_2,\quad T_3\mapsto \hat T_3, \quad T_4 \mapsto \xx^{-1}\hat T_3^{-1}
\end{equation}
extend to an injective algebra homomorphism $\H_{q,\ult}\hookrightarrow \End_\C(\C[X^{\pm 1}])$. 
\end{proposition}

The Dunkl-type embedding above can be viewed as a map from $\H_{q,\ult}$ to a localization of $A_q\rtimes \Z_2$. (Here $A_q$ is the quantum torus, which is generated by $\xx^{\pm 1}$ and $\yy^{\pm 1}$ subject to the relation $\xx \yy = q^2 \yy \xx$, and $\Z_2$ acts on $A_q$ by inverting $\xx$ and $\yy$.) This embedding maps the spherical DAHA into (a localization of) the symmetric algebra $A_q^{\Z_2}$. In the specialization $q = \epsilon = \pm 1$ and $t_3 = t_4 = 1$, a short computation leads to the following description of this symmetric embedding:
\begin{align*}
x_\ult &\mapsto \xx + \xx^{-1}\\
y_\ult &\mapsto t_2(\yy + \yy^{-1}) + \left[ \bar t_2 (\xx \yy^{-1} - \xx^{-1}\yy) + \epsilon \bar t_1 (\yy - \yy^{-1})\right] \delta^{-1}\\
z_\ult &\mapsto \epsilon t_2 (\xx \yy + \xx^{1-}\yy^{-1}) + \left[ \bar t_2 (\yy - \yy^{-1}) + \epsilon \bar t_1(\xx \yy - \xx^{-1}\yy^{-1})\right] \delta^{-1}
\end{align*}
where $\delta := \xx - \xx^{-1}$. If we multiply numerators and denominators by $\delta$ we can rewrite them in terms of $x_T$, $y_T$, and $z_T$, which are the images of the curves $(1,0)$, $(0,1)$, and $(1,1)$ on the torus inside the algebra $A_q^{\Z_2}$. (Note that $z_T = q^{-1}(\xx\yy + \xx^{-1}\yy^{-1})$.) We then obtain the following:
\begin{corollary}\label{cor_symdunkl}
When $q = \epsilon = \pm 1$ and $t_3 = t_4 = 1$, the map $\SH_{\epsilon,t_1,t_2,1,1} \to \sk_\epsilon(T^2)^\loc$ is given by
\begin{align*}
x_\ult &\mapsto x_T\\
y_\ult &\mapsto t_2 y_T + 
\frac{-\bar t_2 (x_T^2 y_T - \epsilon x_T z_T - 2 y_T) + 
\bar t_1 (2 z_T - \epsilon x_T y_T)} 
{x_T^2 - 4}\\
z_\ult &\mapsto t_2 z_T + 
\frac{ \bar t_2( 2z_T - \epsilon x_T y_T) + \bar t_1 (\epsilon x_T z_T - 2y_T)}
{x_T^2-4}
\end{align*}
\end{corollary}

\subsection{Example: the trefoil}\label{sec_extrefoil}
In this subsection we describe Conjecture \ref{conj_bs14} completely explicitly in the example of the trefoil. Several computations were done with the help of the computer (using \texttt{Macaulay2} \cite{M2} and \texttt{Mathematica}).
In this section only we use the abbreviations $S_n := S_n(x_T)$ and $T_n := T_n(x_T)$ for Chebyshev polynomials, which are defined by 
\[
S_n(X+X^{-1}) := \frac{X^{n+1}-X^{-n-1}}{X-X^{-1}},
\quad \quad T_n(X+X^{-1}) = X^n + X^{-n}
\]

Let $K$ be the trefoil in $S^3$. We first recall formulas from \cite{Gel02} for the action of $\SH_{q,\ulone}$ on the skein module $\sk_q(S^3\setminus K)$. (See \cite{BS16} for the conversion into the present notation.) As a module over $\C[x]$ the skein module $\sk_q(S^3\setminus K)$ is generated by two elements $u$ and $v$, and the action of $y_T$ and $z_T$ are given by the following:
\begin{eqnarray*}
y_T\cdot u &=& -(q^2+q^{-2})u\\
z_T\cdot u &=& -q^{-3}S_1u\\
y_T\cdot v &=& (q^6S_4-q^2)u + q^6T_6v\\
z_T\cdot v &=&  q^5S_3u + q^5T_5v
\end{eqnarray*}
The $A$-polynomial of the trefoil is $(L-1)(L+M^{-6})$. (Roughly, this describes the preimage in $(\C^\times)^2$ of the image in $\chr(T^2)$ of the character variety of the knot complement, where $(\C^\times)^2$ maps to $\chr(T^2)$ by sending $(\alpha,\beta)$ to the representation where the generators of $T^2$ are sent to diagonal matrices with upper-left entries $\alpha$ and $\beta$, respectively.) This symmetrizes to the polynomial
\[ 
A_T := (y_T+2)(y_T-T_6)
\]
(remember the negative sign in the map of Theorem \ref{thm_orep}). It is straightforward (at least with a computer) to check that when $q=-1$ the element $A_T \in \sk_{q=-1}(T^2)$ indeed annihilates the module $\sk_{q=-1}(S^3\setminus K)$.

We now give explicit expressions for the action of $y_\ult$ and $z_\ult$ on $u$ and $v$ with the parameter values $q=-1$ and $t_3=t_4=1$. 

\begin{align*}
y_\ult \cdot u &= -(t_2+t_2^{-1}) u\\
z_\ult \cdot u &= [t_2 S_1  - \bar t_1]u\\
y_\ult \cdot v &=  [t_2 (S_4-1) - \bar t_2 (S_4+S_2) + \bar t_1 (S_3+S_1)]u \\
& + [t_2 T_6 - \bar t_2 S_6 + \bar t_1 S_5]v \\
z_\ult \cdot v &= 
[t_2^{-1} S_3  + \bar t_2 S_1 - \bar t_1 (S_2+S_0)]u \\
&+ [-t_2 T_5 + \bar t_2 S_5 - \bar t_1 S_4]v
\end{align*}
It is now possible to compute that the following deformation of the symmetric $A$-polynomial $A_T$ annihilates the skein module:
\begin{equation}\label{eq_Adef}
A_\ult := (y_\ult + t_2 + t_2^{-1})(y_\ult - t_2^{-1}T_6 - \bar t_1 S_5 + \bar t_2 S_4)
\end{equation}

\begin{remark}
Since $\bar t_i = t_i - t_i^{-1}$, it is clear that the $t_1=t_2=1$ specialization of $A_\ult$ is equal to $A_T$. The choice of the element $A_\ult$ in \eqref{eq_Adef} is somewhat arbitrary since there is no canonical choice. This particular choice was made via computer experiment, and it seems to be the simplest obvious deformation of $A_T$ which ``has the same structure.'' However, we remind the reader that the action of $\SH_{q=-1,t_1,t_2,1,1}$ is canonical, even the particular choice of $A_\ult$ is not.
\end{remark}

\subsection{Example: the figure eight}\label{sec_exfig8}
In this subsection we give a explicit description of Conjecture \ref{conj_bs14} in the case of the figure eight knot, again with $q=-1$ and $t_2=t_4=1$, and with the help of a computer. We use the notation $S_n$ and $T_n$ for Chebyshev polynomials as in Section \ref{sec_extrefoil}.

Let $K$ be the figure eight knot and $N := \sk_{q=-1}(S^3\setminus K)$ the skein module of its complement. We recall from \cite{GS04} that as a module over $\C[x]$, the skein module $N$ is freely generated by elements $p$, $u$, and $v$ (again, see \cite{BS16} for conversion into the present notation). Formulas for the action of $y_T$ and $z_T$ at $q=-1$ are given by
\begin{align*}
y_T\cdot p &= -2p\\
y_T\cdot u &= (S_2+1)p + (-T_4+T_2+T_0)u\\
y_T\cdot v &= (-S_2-1)p + (-T_4+T_2+T_0)v\\
z_T\cdot p &= S_1 p\\
z_T\cdot u &= -S_3p +(T_5-T_3-T_1)u + (T_3-T_1)v\\
z_T\cdot v &= 2S_1p + (-T_3+T_1)u + (T_3-2T_1)v
\end{align*}
The $A$-polynomial for the figure eight knot is $(L-1)(L+L^{-1}+-M^4+M^2+2+M^{-2}-M^{-4})$, and a symmetric version is given by
\begin{equation}\label{eq_Afig8}
A := (y_T+2)(y_T+T_4-T_2-T_0)
\end{equation}
(Again, the apparent change in signs is explained by the signs in Theorem \ref{thm_orep}.)
One may now compute that the action of $y_\ult \in \SH_{q=-1,t_1,t_2}$ is given by the formulas
\begin{align}
y_\ult\cdot p &= -(t_2+t_2^{-1})p\notag\\
y_\ult\cdot u &= [t_2(S_2+1) - \bar t_1 S_1]p\notag\\
&+ [t_2(-T_4+T_2+T_0) - \bar t_2 S_2 + \bto T_3]u\notag\\
&+ [-\btt (S_2+1) + 2\bto S_1]v\notag\\
&=: ap+bu+cv\label{eq_consts1}\\
y_\ult\cdot v &= [-t_2^{-1}(S_2+1)-\bto S_1]p\notag\\
&+ [\btt (S_2+1)-2\bto S_1]u\notag\\
&+[t_2^{-1}(-T_4+T_2+T_0)+\btt S_2-\bto T_3]v\notag\\
&=: dp+eu+fv\label{eq_consts2}
\end{align}
Similarly, the action of $z_\ult$ is given by the formulas
\begin{align*}
z_\ult\cdot p &= [t_2 S_1 - \bto]p\\
z_\ult\cdot u &= [-t_2 S_3+\bto(S_2+1)-\btt S_1]p\\
&+ [t_2(T_5-T_3-T_1)+\btt(T_3)+\bto(-T_4+1)]u\\
&+[t_2(T_3-T_1)+2\btt S_1-\btt(S_2+1)] v\\
z_{\ult}\cdot v &= [(t_2+t_2^{-1})S_1] p\\
&+[t_2(-T_3+T_1)-2\btt S_1+\bto(S_2+1)]u\\
&+[t_2^{-1}(T_3-2T_1)-2\btt S_1+\bto S_2]v
\end{align*}
As a sanity check, one may check directly (or by computer) that the above formulas satisfy the cubic relation of Theorem \ref{thm_terwilliger} (specialized to $q=-1$ and $t_3=t_4=1$). One can also check that the following element annihilates the skein module $N$:
\begin{equation}
\tilde A := (y_\ult+t_2+t_2^{-1})(y_\ult^2 - (b + f)y_\ult + (bf-ce))
\end{equation}
where the constants in the formula were defined in equations \eqref{eq_consts1} and \eqref{eq_consts2}. In fact, it is obvious that $\tilde A$ annihilates $N$: the element $p$ generates a submodule of $N$ annihilated by $y+t_2+t_2^{-1}$, and the second factor in the definition of $\tilde A$ is just the characteristic polynomial of $y_\ult$, viewed as an operator on the $\C[x]$-module $N / p$.

\begin{remark}
Experimentally, the polynomial $A$ defined in \eqref{eq_Afig8} does not seem to have a deformation which annihilates $N$. (In particular, the $bf-ce$ term doesn't factor unless $t_1=t_2=1$.) However, one can check that if we specialize $t_1=t_2=1$, then 
\[
\tilde A_{t_1=t_2=1} = (y+2)(y+T_4-T_2-T_0)^2
\]
In particular, in this specialization the scheme defined by $\tilde A$ is the same as that defined by $A$, except that one component has been ``fattened.''
\end{remark}

\section{Deformations of the peripheral map and the Brumfiel-Hilden algebra}\label{sec_bh}

Let $K\subset S^3$ be a knot, $M := S^3\setminus K$ its complement, $\pi := \pi_1(M)$ its fundamental group. The fundamental group of the torus boundary of $M$ maps to $\pi$ via the \emph{peripheral map}:
\[
\Z \stackrel{\alpha_m}{\longrightarrow} \Z^2 \stackrel{\alpha}{\longrightarrow} \pi
\]
where we have fixed generators of $\Z^2$ to be the standard longitude and meridian of $K$, and where the image of the generator under $\alpha_m$ is the meridian. By Corollary \ref{cor_toruschar} and Theorem \ref{thm_orep} the peripheral map induces the following map of commutative algebras:
\begin{equation*}
 \alpha_*: \SH_{q=-1,t_i=1} \to \ochar(\pi_1(M))
\end{equation*}
where $\ochar(\pi_1(M)) := \C[\Rep(\pi_1(M),\SL_2(\C))]^{\SL_2(\C)}$ is the ring of functions on the character variety of the knot complement.
\begin{lemma}
 For $q = -1$, the action of $\SH_{q,\ult}$ on $\sk_q(M)$ conjectured in \cite[Conj. 1]{BS16} arises from an algebra homomorphism
 \[
  (\alpha_t)_*: \SH_{-1,t_1,t_2,1,1} \to \ochar(\pi_1(M))
 \]
which we call a \emph{deformed peripheral map}.
\end{lemma}
\begin{proof}
 Consider the commutative diagram
\begin{equation}\label{dia_qeq1}
\begin{diagram}
 & & \ochar(\partial M) & \rTo^{\alpha_*} & \ochar(M) \\
 &  & \dTo^\loc &  & \dTo^\loc \\
\SH_{-1,\ult} & \rInto^{\Phi_{-1,\ult} } & \ochar(\partial M)^\loc & \rTo^{\alpha_*^\loc} & \ochar(M)^\loc
\end{diagram}
\end{equation}
which is obtained from the diagram of $\sk_q(\partial M)$-modules by specializing $q=-1$:
\begin{diagram}
 & & \sk_q(\partial M) & \rTo & \sk_q(M)\\
 &  & \dTo^\loc &  & \dTo^\loc\\
\SH_{q,\ult} & \rInto & \sk_q(\partial M)^\loc & \rTo & \sk_q(M)^\loc
\end{diagram}
(The top horizontal map is given by $a \mapsto a\cdot \varnothing$, which is a map of left modules for general $q$ and a map of commutative algebras when $q=-1$.) Conjecture 1 says that $\SH_{q,\ult}[\sk_q(M)] \subset \sk_q(M) \subset \sk_q(M)^\loc$, which implies
\begin{equation}\label{eq_cond3}
 \SH_{q,\ult}\cdot \varnothing \subset \sk_q(M)
\end{equation}
The diagram (\ref{dia_qeq1}) consists of \emph{algebra} homomorphisms, and condition (\ref{eq_cond3}) specializes to 
\[
 \Phi_{-1,\ult}(\SH_{-1,\ult})\cdot 1 \subset \ochar(M)
\]
This shows that $\Phi_{-1,\ult}:\SH_{-1,\ult} \to \ochar(M)$ is an algebra map.
\end{proof}

Geometrically, we thus have a morphism of schemes
\[
 \alpha_\ult: \chr(M) \to \mathrm{Spec}(\SH_{-1,\ult})
\]
that is a deformation of the classical restriction map. 

\subsection{The Brumfiel-Hilden conjecture}
\la{BHC}
We now recall the definition of the Brumfiel-Hilden algebras  from \cite{BH95}. 
If $\pi$ is a finitely generated discrete group, these algebras are defined as follows:
\begin{equation}\label{bhdef}
H[\pi] := \frac{\C[\pi]}{\langle g(h+h^{-1}) - (h+h^{-1})g\rangle},\quad \quad 
H^+[\pi] := \langle g + g^{-1} \mid g \in \pi\rangle \subset H[\pi] 
\end{equation}
A conceptual explanation for these definitions is given by the following theorem.
\begin{theorem}[{\cite{BH95}}]\label{bhthm}
If $\pi$ is a finitely generated group, then
\begin{enumerate}
\item 
The commutative algebra $H^+[\pi]$ is isomorphic to $\O(\chr(\pi))$, the ring of functions on the $\SL_2(\C)$ character variety of $\pi$. 
\item The algebra $H[\pi]$ is isomorphic to the ring $\Gamma(\mathrm{Rep} (\pi),M_2(\C))^{\SL_2(\C)}$ of $\SL_2(\C)$-equivariant matrix-valued functions on the $\SL_2(\C)$ representation variety of $\pi$.
\end{enumerate}
\end{theorem}
\begin{remark}
The first map sends $g+g^{-1}$ to the function $\rho \mapsto \Tr(\rho(g+g^{-1})$. The second sends $g$ to the matrix-valued function $\rho \mapsto \rho(g)$, which is well-defined because if $A \in \SL_2$ then $A+A^{-1} = \Tr(A) \mathrm{Id}$ is central in the ring of $2\times 2$ matrices.
\end{remark}
We also recall the \emph{Brumfiel-Hilden conjecture}, see \cite[pg. 122]{BH95}:
\begin{conjecture}[\cite{BH95}]\label{bhconj}
Let $\pi = \pi_1(S^3\setminus K)$ be the fundamental group of the complement of a knot in $S^3$, and let $X,Y \in H[\pi]$ be the standard meridian and longitude of $K$. Then
\begin{equation}\label{bhcond}
Y \in H^+[X^{\pm 1}]
\end{equation}
where the right hand side is the subalgebra of $H[\pi]$ generated by $H^+[\pi]$ and the elements $X^{\pm 1} \in H[\pi]$.
\end{conjecture}

Our goal is to relate this conjecture to the main conjecture of \cite{BS16}. To shorten notation, we write $H^+ := H^+[\pi]$ and $H := H[\pi]$, and we write 
\[
\delta := X - X^{-1}
\]
We also define an operator $s: H \to H$ by the formula $s\cdot g := g^{-1}$.
Finally, we define the following $H^+$-module:
\begin{equation}\label{eq_defN}
 N := H^+[X^\pmone] + H^+[X^{\pm 1}](Y+1)\delta^{-1} \subset H[\delta^{-1}]
\end{equation}
To clarify, $H[\delta^{-1}]$ is considered as an $H$-module, and \emph{not} as an algebra (so that we don't have to deal with the issue of localizing noncommutative algebras). More formally, we define $H[\delta^{-1}]$ as $H \otimes_{\c[X^{\pm 1}]} \c[X^{\pm 1}]\delta$, where the right hand term of the tensor product is the $\C[X^{\pm 1}]$ submodule of the algebra $\C[X^{\pm 1},\delta^{-1}]$ generated by $\delta^{-1}$. 
The action of $s$ on $H$ clearly extends to $H[\delta^{-1}]$ via $s\cdot \delta = -\delta$.

\begin{proposition}\label{prop_3conditions}
 The following conditions are equivalent:
 \begin{enumerate}
  \item $Y-Y^{-1} \in H^+\delta$\label{i1}
  \item $Y \in H^+[X^{\pm 1}]$\label{i2}
  \item \label{i3} $\e N = H^+$ and
  $(s+Y)\cdot N \subset \delta N$, where $\e := (1+s)/2$ and $s$ acts on $N$ as in the previous paragraph.
 \end{enumerate}
\end{proposition}
\begin{proof}
We first prove that (\ref{i1}) is equivalent to (\ref{i2}). If $Y-Y^{-1} \in H^+\delta$ for some $A \in H^+$, then $Y = (Y-Y^{-1})/2 + (Y+Y^{-1})/2 \in H^+\delta + H^+ = H^+[X^{\pm 1}]$. Similarly, if $Y \in H^+[X^{\pm 1}]$, then $Y = A + B\delta$ for some $A,B \in H^+$. Then applying the anti-involution $g \mapsto g^{-1}$, we obtain $Y^{-1} = A - \delta B = A - B \delta$. This implies $Y - Y^{-1} = B\delta$.

Next we show that (\ref{i1}) implies (\ref{i3}). We have that $H^+[X^{\pm 1}] = H^+ + H^+\delta$, which implies 
\[
N = H^+ + H^+\delta + H^+(Y+1)\delta^{-1} + H^+(Y+1) = H^+ + H^+\delta + H^+(Y+1)\delta^{-1}
\]
where the last term of the second expression is contained in $H^+[X^{\pm 1}] \subset N$ by condition (\ref{i2}) (which is implied by condition (\ref{i1})). Since  $s\delta = -\delta$, we see that $\e\delta = \e\delta^{-1} = 0$. Therefore, after multiplying this expression by $\e$ we obtain
\begin{eqnarray*}
\e N &=& H^+ + H^+(\e Y + \e)\delta^{-1}\\
&=& H^+ + H^+(Y + Y^{-1}s)\delta^{-1}\\
&=& H^+ + H^+(Y-Y^{-1})\delta^{-1}\\
&=& H^+
\end{eqnarray*}
where the last equality follows from assumption (\ref{i1}) that $Y-Y^{-1} \in H^+\delta$. Next, we compute
\begin{eqnarray*}
(s+Y) \cdot H^+ &=& (1+Y)\cdot H^+ = \delta\left[(1+Y)H^+\delta^{-1}\right] \subset \delta N \\
(s+Y)\cdot H^+\delta &=& (-1 + Y)H^+\delta \subset H^+[X^{\pm 1}]\delta \subset \delta N \\
(s+Y)H^+(Y+1)\delta^{-1} &= & H^+(Y^2+Y-Y^{-1}-1)\delta^{-1} \\
&=& H^+(Y-Y^{-1})(Y+1)\delta^{-1} \subset H^+\delta(Y+1)\delta^{-1} \subset \delta N
\end{eqnarray*}

Finally, we show that condition (\ref{i3}) implies condition (\ref{i2}). Acting on the assumed containment by $1-s$ we obtain
\[
(1-s)(s+Y)N \subset (1-s)\delta N = \delta(1+s)N = \delta H^+
\]
We then compute
\[
(1-s)(s+Y)\e N = (s+Y-1-Y^{-1}s)\e N = (Y-Y^{-1}) H^+ 
\]
This shows that $(Y-Y^{-1})H^+ \subset \delta H^+$, which completes the proof.
\end{proof}

\begin{corollary}\label{cor_BH_def}
Suppose $K$ is a knot that satisfies the (equivalent) conditions of Proposition \ref{prop_3conditions}. Furthermore, assume that the map $\delta: N \to N$ given by multiplication by $\delta$ 
is injective. Then Conjecture \ref{conj_bs14} (at $q=-1$) holds for $K$.
\end{corollary}
\begin{proof}
The algebra $\H_{q=-1,t_1,t_2,1,1}$ is generated by $X^{\pm 1}$, the operator $T_2$, and the involution $s$. Under the Dunkl embedding \eqref{ccdunklembedding}, the poles of the image of the element $T_2$ are of the form $\delta^{-1}(s+Y)$. Therefore,  third condition in Proposition \ref{prop_3conditions} together with the injectivity assumption implies the operator $T_2$ acts on the module $N$. Then the spherical subalgebra $\e \H_{-1,t_1,t_2,1,1}\e$ acts on $\e N$, which is equal to $\H^+$ (again by the third condition in Proposition \ref{prop_3conditions}). This confirms Conjecture \ref{conj_bs14} at $q=-1$.
\end{proof}

\section{Relations between different knots}\label{sec_connectsum}
In this section, we describe implications for the Brumfiel-Hilden condition \eqref{bhcond} between different knots. We first recall that the \emph{connect sum} $K \# K'$ of two knots is defined by attaching two points, one on each knot, and then resolving the double point to obtain one knot. (See, e.g. \cite{BZ03}.)

\begin{lemma}\label{lemma_connectsum}
If knots $K$ and $K'$ satisfy (\ref{bhcond}), then the connect sum $K \# K'$ does too.
\end{lemma}
\begin{proof}
Let $\pi, \pi'$ be the knot groups of $K,K'$, respectively, with peripheral systems $m,l,m',l'$, respectively. Then \cite[Prop. 7.10]{BZ03} says that $\pi_1(K\# K') = \pi \ast_\Z \pi'$, where the generator of $\Z$ maps to $m$ and $m'$ inside $\pi$ and $\pi'$, respectively. In particular, the images of $m$ and $m'$ inside of $\pi_1(K \# K)$ are equal. Next, a peripheral system of $K\# K'$ is given by $(m,ll')$ (or by $(m',ll')$), where we have abused notation by writing $l$ and $l'$ for their images inside $\pi_1(K\# K')$. This is true because one can write the longitude of a knot $K$ in terms of the Wirtinger generators by ``tracing along the knot and recording crossings,'' and in the connect sum, one can first trace along $K$ and then along $K'$. By assumption, we have $Y_K \in H_n^+[K][X_K^{\pm 1}]$ and $Y_{K'} \in H_n^+[K'][X_{K'}^{\pm 1}]$, and combining these statements with the presentation of $\pi_1[K \# K']$ shows that $Y_K,\, Y_{K'} \in H^+[K \# K'][X_{K \# K'}^{\pm 1}]$. Finally, since $Y_{K\#K'} = Y_KY_{K'}$, this shows that $Y_{K\# K'} \in H_n^+[K\# K'][X_{K\# K'}^{\pm 1}]$, which is what we wanted.
\end{proof}

We will write $K' \geq_p K$ if there is a surjection $f: \pi_1(K') \twoheadrightarrow \pi_1(K)$ which preserves the peripherial systems: in other words, $f(X_{K'}) = X_{K}$ and $f(Y_{K'}) = Y_{K}^d$ for some $d \in \Z$. (Such surjections are not common, but they will be useful for our purposes for torus knots.)

\begin{lemma}\label{lemma_covering}
Suppose for a knot $K$ there exist knots $K_i \subset S^3$ satisfying (\ref{bhcond}) and that $K_i \geq_p K$ for each $i$, with $Y_i \mapsto Y_K^{d_i}$. Further suppose that the integers $d_i \in \Z$ generate $\Z$ as a group. Then $K$ satisfies (\ref{bhcond}).
\end{lemma}
\begin{proof}
The definitions of $H$ and $H^+$ in (\ref{bhdef}) are functorial. In particular a surjection $f: \pi_1(K_i) \twoheadrightarrow \pi_1(K)$ induces a surjection $f_*:H[K_i] \twoheadrightarrow H[K]$ with $f_*(H^+[K_i]) = H^+[K]$. Since $f$ preserves the peripheral systems and we have assumed $Y_{i} \in H^+[K_i][X_i^{\pm 1}]$, this shows that $Y_{K}^{d_i} \in H^+[K][X^{\pm 1}_{K}]$. By assumption, there exist $a_i \in \Z$ such that $\sum_i a_i d_i = 1$. This implies that $Y_K^{\sum a_i d_i} = Y_K \in H^+[X_K^{\pm 1}]$, which completes the proof.
\end{proof}

\section{Torus knots}\label{sec_torus}
In this section, we will prove Conjecture \ref{bhconj} for torus knots. 
To this end, we will use the presentation of the BH algebra for two-generator groups
given in the Appendix.
Let $r$ and $s$ be coprime integers such that $\, 2 \leq r < s \,$. The knot group
of the $(r,s)$-torus knot $ K = K(r,s) $ has the following presentation
$$
\pi(K) = \langle u,v\mid u^r = v^s \rangle\ ,
$$
and the meridian and longitude are represented by the elements (cf. \cite[Prop. 3.28]{BZ03}):
\begin{equation}
\la{torus}
m = u^n v^{-k}\ ,\quad l = v^s m^{-rs}
\end{equation}
where $k$ and $n$ are integers satisfying 
\begin{equation}
\la{nk}
- r k + s n = 1\ .
\end{equation}
We remark that $m$ is independent of the choice of solution $(k,n)$ to the equation \eqref{nk}.

Now, let $ H = H[\pi] $ be the Brumfiel-Hilden algebra of the knot group $\pi(K)$.
As in Section~\ref{BHC}, let $X$ and $Y$ denote the images of the elements 
$ m $ and $ l $ in $ H $ under the canonical projection $ \c[\pi] \onto H[\pi] $, and let
$ H^{+}[X^{\pm 1}] $ be the subalgebra of $H$ generated by $ H^+ $ and
$ \c[X^{\pm 1}] $.  Recall that Conjecture~\ref{bhconj}  is equivalent to the
statement
\begin{equation}
\la{BHconj}
Y \in H^{+}[X^{\pm 1}]\ .
\end{equation}
The main result of this section is 
\begin{theorem}
\la{Ttorus}
Condition \eqref{BHconj} holds for all torus knots.
\end{theorem}

We will prove Theorem~\ref{Ttorus} in several steps. First, we verify 
\eqref{BHconj} for $(p,p+1)$-torus knots by direct calculation. Then, given
a torus knot $K(r,s)$ with $rs$ even, we construct a group epimorphism $ \pi[K(p,p+1)] 
\onto \pi[K(r,s)] $ and we use Lemma \ref{lemma_covering} to give a covering argument to show that \eqref{BHconj} holds for $ K(r,s) $, provided it holds for $ K(p,p+1)$ for all $p$. Next, we show that \eqref{BHconj} holds for $(2,2p+1)$ torus knots and use a similar covering argument to show the same for $K(r,s)$ with $rs$ odd.

In the computations below we will use the classical Chebyshev polynomials of the first
and second kind. We recall that these polynomials are defined respectively by 
\begin{eqnarray*}
 T_1(y) &=& 1,\quad \quad T_3(y) = y,\quad\quad \,\,\,\,T_{n+1} = 2yT_n - T_{n-1}\\
 U_0(y) &=& 1,\quad \quad U_1(y) = 2y,\quad \quad U_{n+1} = 2yU_n - U_{n-1}\ .
\end{eqnarray*}
Alternative definitions are given by
\begin{equation}
\la{chebtr}
T_n(\cos(\vartheta)) := \cos(n \vartheta)\ ,\quad
U_n(\cos(\vartheta)) := \frac{\sin((n+1)\vartheta)}{\sin(\vartheta)}\ ,
\quad n = 0,\,1,\,2,\,\ldots
\end{equation}

\subsection{Torus knots of type $(p,p+1)$}
This section is devoted to the proof of of the following proposition.
\begin{proposition}\label{prop_ppp1}
Condition \eqref{BHconj} holds for $(p,p+1)$ torus knots.
\end{proposition}

Fix an integer $p \geq 2$ and consider the torus knot $ K(p,p+1)$. If $r=p $ and $ s = p+1 $,
we can take $ k = n = 1 $ in \eqref{nk}, so that the peripheral elements \eqref{torus} are given by 
\begin{equation}\label{eq_mlpp1}
m = uv^{-1}\ ,\quad l = v^{p+1}m^{-p(p+1)}
\end{equation}
To do calculations it is convenient to change generators of the knot group taking
$\,a := uv^{-1}\,$ and $\,b=v\,$. Then
\begin{equation}\label{eq_grouprelation}
\pi(K(p,p+1)) = \langle u,v\mid u^p = v^{p+1} \rangle = 
\langle a,b \mid abab \ldots aba = b^p \rangle\ ,
\end{equation}
where there are $(p-1)$ copies of $b$ on the left-hand side. The peripheral pair becomes
\[
m = a\ ,\quad l = b^{p+1}a^{-p(p+1)}\ .
\]

Next, recall the presentation of the BH algebra of the free group 
$ F_2 = \langle a,b \rangle $ given in the Appendix:
$$
H[F_2] \cong R \oplus Rt\ ,\quad t^2 = y^2-1
$$
where $ R = \c[X^{\pm 1}, y, z]$, and where $X = a$, $y = (b+b^{-1})/2$, $2z = ab + b^{-1}a^{-1} + (a+a^{-1})y$, and $t = (b-b^{-1})/2$. We will repeatedly use the following simple observation.
\begin{lemma}
\label{L1}
For any word $ c \in F_2 $, we have 
$$
c^{n+1} = U_n(c^+)\,c - U_{n-1}(c^+)\ ,\quad n =1,\,2,\,\ldots
$$
where $ c^{+} := (c+c^{-1})/2 $.
\end{lemma}
\begin{proof}
For $n=1$, we have
$\,c^2 = 2 [(c+c^{-1})/2]\,c - 1 = 2 c^+ c - 1 = U_1(c^+)c - U_0(c^+)\,$,
and for all $ n\ge 1 $, the claim follows easily by induction in $ n $. 
\end{proof}
\begin{remark}
Note that the identity of Lemma~\ref{L1} holds actually in the group algebra
$ \c[F_2] $ but we will use it as an identity in $ H[F_2] $.
\end{remark}

Using Lemma \ref{L1}, we compute the images of the left and right hand sides of the relation (\ref{eq_grouprelation}) in $H[F_2]$:
\begin{eqnarray*}
abab\ldots a &=& (ab)^pb^{-1}\\
&=& U_{p-1}(Q)ab - U_{p-2}(Q)b^{-1}\\
&=& U_{p-1}(Q) a - U_{p-2}(Q) b^{-1}\\
&=& U_{p-1}(Q) X - U_{p-2}(Q)y + U_{p-2}(Q)t
\end{eqnarray*}
where $Q = (ab + b^{-1}a^{-1})/2 = (X(y+t) + (y-t)X^{-1})/2 = xy+z$. Second, we compute
\begin{eqnarray}
b^p &=& U_{p-1}(y)b - U_{p-2}(y)\label{yu_2_bp}\\
&=& U_{p-1}(y) y - U_{p-2}(y) + U_{p-1}(y) t\notag\\
&=& T_p(y) + U_{p-1}(y)t\notag
\end{eqnarray}
Hence, $abab\ldots a - b^p = A + Bt$, where
\begin{eqnarray*}
A &:=& U_{p-1}(Q) - U_{p-2}(Q)y - T_p(y)\\
B &:=& U_{p-2}(Q) - U_{p-1}(y)
\end{eqnarray*}

We would like to show that the longitude $l$ is in the subalgebra $[R] \subset H[\pi]$ which is the image of $R$ under the quotient $H[F_2] \twoheadrightarrow H[\pi]$. By Proposition \ref{yuri_prop2}, this is true if $l = [r_0 + rt]$ with $r \in J$, where $J$ is the ideal
\begin{eqnarray}\label{eq_olddefofJ}
J_1 &=& \langle A,A^\sigma,A^\delta,A^{\sigma \delta},B\rangle
\end{eqnarray}
Here we have written $A^\sigma = \sigma(A)$, etc. where $\sigma: R \to R$ is the involution of $R$ and $\delta: R \to R$ is the $\sigma$-derivation of $R$ defined by 
\begin{align*}
\sigma(X^{\pm 1}) &= X^{\mp 1},\quad \quad \sigma(y) = y,\quad \quad\sigma(z)= z\\
\delta(X^{\pm 1}) &= \pm 2z,\quad \quad  \delta(y)=\delta(z) = 0
\end{align*}
Note that since $U_{p-1}(y) y + T_p(y) = U_p(y)$, we can rewrite $A$ as 
\[
A = U_{p-1}(Q) X - B y - U_p(y)
\]
Hence,
\begin{eqnarray*}
A^\sigma &=& X^{-1}U_{p-1}(Q) - By-U_p(y)\\
A^\delta &=& 2U_{p-1}(Q)z\\
A^{\sigma \delta} &=& -2U_{p-1}(Q)z
\end{eqnarray*}
It follows that $J_1\subset R$ is defined by 
\[
J_1 = \langle U_{p-1}(Q)X - U_p(y),U_{p-1}(Q)X^{-1}-U_p(y),U_{p-2}(Q)-U_{p-1}(y),zU_{p-1}(Q)\rangle
\]
Since $Q = z + xy$, we have
\begin{eqnarray*}
QU_{p-1}(Q) &=& zU_{p-1}(Q) + xy U_{p-1}(Q)\\
&=& zU_{p-1}(Q) + y(XU_{p-1}(Q) + X^{-1}U_{p-1}(Q))/2\\
&\equiv& y U_{p}(y) \,\,(\mathrm{mod}\,\,J_1)
\end{eqnarray*}
This shows that $J_1$ is generated by the elements
\begin{equation}\label{yu_relations}
U_{p-1}(Q)X - U_p(y),\quad U_{p-1}(Q)X^{-1} - U_p(y),\quad U_{p-2}(Q) - U_{p-1}(y),\quad QU_{p-1}(Q) - yU_p(y)
\end{equation}
Since $Y=b^{p+1}a^{-p(p+1)} = b^{p+1}x^{-p(p+1)} \in H[\pi]$, to verify (\ref{BHconj}) it suffices to show that $b^{p+1} \in H^+[X^{\pm 1}]$. For this, by 
the same computation as in \eqref{yu_2_bp}, it suffices to show that 
\begin{equation}\label{yu_9}
U_p(y) \in J_1
\end{equation}
We will need some elementary properties of Chebyshev polynomials, which we give in the following:
\begin{lemma}\label{yu_lemma2}
For any $p \geq 2$, we have
\begin{eqnarray*}
U_{p-1}U_{p+1}&=& -1 + U_p^2\\
\mathrm{gcd}(U_p-U_{p-1},U_{p-2}-U_{p-3}) &=& 1
\end{eqnarray*}
\end{lemma}

To simplify the notation, we set
\[
E := U_{p-1}(Q),\quad F:=U_{p-2}(Q)-U_{p-1}(y),\quad N := U_p(y)
\]
We also write ``$\equiv$'' for congruences in $R$ modulo $J_1$. The relations (\ref{yu_relations}) imply
\begin{eqnarray}
F &\equiv& 0\label{yu_1}\\
E &\equiv& XN \label{yu_2}\\
E &\equiv& X^{-1}N \label{yu_3}\\
QE &\equiv & yN \label{yu_4}
\end{eqnarray}
We need to show that $N \equiv 0$. By \eqref{yu_1} and Lemma \ref{yu_lemma2}, we have
\begin{eqnarray}
U_{p-3}(Q)E &=& U_{p-3}(Q)U_{p-1}(Q)\label{yu_6}\\
&=& -1 + U^2_{p-2}(Q)\notag\\
&\equiv& -1+U^2_{p-1}(y)\notag\\
&=& U_{p-2}(y)U_{p-1}(y)\notag\\
&=& U_{p-2}(y)N\notag
\end{eqnarray}
By \eqref{yu_2} and \eqref{yu_4}, 
\begin{equation}\label{yu_7}
QN \equiv (X^{-1}y)N
\end{equation}
Now if we combine (\ref{yu_3}), \eqref{yu_6}, and (\ref{yu_7}), we get
\begin{equation}\label{yu_8}
U_{p-2}(y)N \equiv U_{p-3}(Q) E \equiv X U_{p-3}(Q) N \equiv X U_{p-3}(X^{-1}y)N
\end{equation}
Now assume that $p$ is even. Then $U_{p-3}$ is an odd polynomial. Equations (\ref{yu_2}) and (\ref{yu_3}) show that $X^2N \equiv N$, which implies
\begin{equation}\label{yu_9}
XU_{p-3}(X^{-1}y)N \equiv U_{p-3}(y) N
\end{equation}
It therefore follows from (\ref{yu_8}) that
\begin{equation}\label{yu_10}
\left[U_{p-3}(y) - U_{p-2}(y)\right] N \equiv 0
\end{equation}
Similarly, by (\ref{yu_4}), $U_{p-1}(Q)N \equiv X^{-1}U_p(y)N$, which by (\ref{yu_7}) implies
\begin{equation}\label{yu_11}
XU_{p-1}(X^{-1}y)N \equiv U_p(y)N
\end{equation}
Again, if $p$ is even, then $U_{p-1}$ is an odd polynomial, so that $XU_{p-1}(X^{-1}y)N \equiv U_{p-1}(y)N$. Hence, (\ref{yu_11}) becomes
\begin{equation}\label{yu_12}
\left[ U_{p-1}(y) - U_p(y)\right]N \equiv 0
\end{equation}
By Lemma \ref{yu_lemma2}, the polynomials $U_{p-3}-U_{p-2}$ and $U_{p-1}-U_p$ are relatively prime. Hence, if $p$ is even, equations (\ref{yu_10}) and (\ref{yu_12}) combined together imply
\[
N \equiv 0
\]
Now, if $p$ is odd, arguing in a similar fashion, we can also derive from (\ref{yu_6}) the relations
\begin{eqnarray*}
\left(U_{p-3}(Q) - U_{p-2}(Q)\right) N &\equiv & 0\\
\left( U_p(Q) - U_{p-1}(Q)\right) N &\equiv& 0
\end{eqnarray*}
where the Chebyshev polynomials depend on $Q$ rather than $y$. By Lemma \ref{yu_lemma2}, we again conclude that $N \equiv 0$. Thus, for all $p \geq 2$ we have $N \equiv 0$, which completes the proof of Theorem \ref{Ttorus} for $(p,p+1)$ torus knots.

\subsection{Torus knots with $rs$ even}\label{sec_rseven}
We will now deduce Theorem \ref{Ttorus} for the $K(r,s)$ torus knot with $rs$ even using Propsition \ref{prop_ppp1} about $(p,p+1)$ torus knots combined with a covering argument using Lemma \ref{lemma_covering}. We first note that given relatively prime $r,s \in \Z$ there exist $n,k \in \Z$ such that
\[
-rk + sn = 1
\]
Let $p = rk$, and let $\tilde K$ be the $(p,p+1)$ torus knot with generators $\tilde u$ and $\tilde v$ satisfying $\tilde u^{p} = \tilde v^{p+1}$.

\begin{lemma}\label{lemma_ppp1covering}
Let $K$ be the $(r,s)$ torus knot with generators $u,v \in \pi_1(K)$ satisfying $u^r = v^s$ and with meridian and longitude $m = u^n v^{-k}$ and $l = v^s m^{-rs}$ as in (\ref{torus}). Then there is a covering map $\pi_1(\tilde K) \twoheadrightarrow \pi_1(K)$ sending $\tilde m \mapsto m$ and $\tilde l \mapsto l^{kn}$.
\end{lemma}
\begin{proof}
We construct the claimed covering map directly via $\tilde u \mapsto u^n$ and $\tilde v \mapsto v^k$. We then check 
\[
\tilde u^p \mapsto u^{pn} = (u^{r})^{kn}= (v^s)^{kn} = (v^k)^{sn} = (\mathrm{im}(\tilde v))^{p+1}                                                                                                                                                                                                                                       \]
which shows that this map is well-defined. It is surjective because $u^r = v^s$ and because $n,k$ are relatively prime. We then note that by equation (\ref{eq_mlpp1}), the meridian and longitude for $\tilde K$ satisfy $\tilde m = \tilde u \tilde v^{-1} $ and $ \tilde l = \tilde v^{p+1}\tilde m^{-p(p+1)}$. We then check
\[
\tilde m \mapsto u^n v^{-k} = m
\]
Similarly, we compute
\[
\tilde l = \tilde v^{p+1} \tilde m^{-p(p+1)} \mapsto v^{k(p+1)} m^{-p(p+1)}
= v^{ksn} m^{-rksn} = (v^s m^{-rs})^{kn} = l^{kn}
\]
(We remark that the second to last equality follows from the fact that $l$ commutes with $m$, which implies that $v^s$ commutes with $m$ also.) 
\end{proof}

In the previous lemma we only used one solution $(k,n)$ to the equation $-rk + sn = 1$. However, an arbitrary solution to this equation is given by
\[
-r(k + ts) + s(n + tr) = 1,\quad \quad t \in \Z
\]
\begin{lemma}\label{lemma_numbertheory}
Let $N(t) = (k+ts)(n+tr)$, and suppose that $rs$ is even. Then the ideal in $\Z$ generated by the set $\{N(t)\}$ is equal to $\Z$.
\end{lemma}
\begin{proof}
Let $M \subset \Z$ be the $\Z$-submodule generated by $N(t)$, and let $c = nk$, $b = sn+rk$, and $a = rs$, so that $N(t) = at^2 + bt + c$. It is clear that $c \in M$, and since $N(-1) + N(1) - 2N(0) = 2a$, we see that $2a \in M$, and similarly $2b \in M$. Finally, $N(2) - N(-1) = 3a + 3b$, which shows that $a+b \in M$. Summarizing, $\langle 2a, 2b, a+b,c\rangle \subset M$.

Now we claim that $a$ and $b$ are relatively prime. Suppose not, so that there is a prime $p$ dividing $a = rs$ and $b = sn + rk$. Since $r$ and $s$ are relatively prime, $p$ must divide either $r$ or $s$. However, $b = sn + rk = 2rk + 1$, which means that $p$ cannot divide both $r$ and $b$. Similarly, $b = 2sn - 1$, which means that $p$ cannot divide both $s$ and $b$, which is contradiction. 

Since $2a,2b \in M$ and $a,b$ are relatively prime, this shows $2 \in M$. Now we have assumed that $a = rs$ is even, which implies $a + b = a + 2rk + 1$ is odd. Since $2 \in M$ and $a+b \in M$, this shows that $1 \in M$, which completes the proof.
\end{proof}

\begin{corollary}
If $rs \in \Z$ is even, then the $(r,s)$ torus knot satisfies condition \eqref{BHconj}.
\end{corollary}
\begin{proof}
This follows from Lemma \ref{lemma_numbertheory} and Lemma \ref{lemma_covering}.
\end{proof}

\subsection{Torus knots with $rs$ odd}
In this section we use a covering argument similar to the one in the previous section to prove that condition \eqref{BHconj} for $(r,s)$ torus knots with $rs$ odd follows from the same conjecture for $(p,2p+1)$ torus knots. We then show this condition holds for $(p,2p+1)$ torus using some calculations along with results of the previous section.

\subsubsection{Covering}
Assume that $r,s \in \Z$ are both odd and are relatively prime. Then there exist $\tilde k, \, \tilde n \in \Z$ such that 
\[
-\tilde k r + \tilde n s = 1
\]
Since $r$ and $s$ are both odd, one of $\tilde k$ or $\tilde n$ must be even. Assume without loss of generality that $\tilde k =: 2k$ is even, and let $n := \tilde n$. Then we have
\begin{equation}\label{eq_2star}
-2kr + ns = 1
\end{equation}
Define $p := kr$ and $q := ns = 2p+1$. Arguing similarly to Lemma \ref{lemma_ppp1covering}, we have a group epimorphism
\[
\pi_1(p,2p+1) \twoheadrightarrow \pi_1(r,s),\quad \quad \tilde m \mapsto m,\quad \tilde l \mapsto l^{nk}
\]
Given a fixed choice of $(n,k)$ satisfying (\ref{eq_2star}), any possible choice $(n',k')$ satisfying the same equation is given by
\[
k' := k + st,\quad n' := n + 2rt,\quad \quad t \in \Z
\]
If we define $N(t) := n'k' = nk + (2rk + ns)t + 2rst^2$, then $\tilde l \mapsto l^{N(t)}$. Consider the ideal in $\Z$ generated by the values of $N(t)$:
\[
I_{r,s} := \langle N(t) \mid t \in \Z\rangle \subset \Z
\]
\begin{lemma}
We have the equality $I_{r,s} = \Z$.
\end{lemma}
\begin{proof}
Put $a := 2rs$, $b = 2rk+ns = 4rk+1 = 2ns - 1$, and $c = nk$. Then 
\[
I_{r,s} = \langle c + bt + at^2 \mid t \in \Z\rangle
\]

From the proof of Lemma \ref{lemma_numbertheory}, we see that $\langle 2a,2b,a+b,c\rangle \subset I_{r,s}$. We then note that $gcd(a,b) = 1$. Indeed, suppose some prime $P$ divides $a = 2rs$. Then $P$ divides either $2$, $r$ or $s$, but in each case $P$ cannot divide $b$ because $b = 4rk+1 = 2ns-1$.

Since $gcd(a,b) = 1$, we have $2 = gcd(2a,2b) \in I_{r,s}$. However, $a = 2rs$ is even and $b = 2ns-1$ is odd, which implies $a+b$ is odd. Therefore, $I_{r,s} = \Z$.
\end{proof}
\begin{corollary}
Condition \eqref{BHconj} for the $(r,s)$ torus knot with $rs$ odd follows from Condition \eqref{BHconj} for the $(p,2p+1)$ torus knots.
\end{corollary}

\subsubsection{$(p,2p+1)$ torus knots}
In this section we prove condition \eqref{BHconj} for $(p,2p+1)$ torus knots. If $p$ is even, then we proved this in Section \ref{sec_rseven}, so we will assume that $p$ is odd.

\begin{proposition}\label{prop_p2pp1podd}
If $p$ is odd, then the $(p,2p+1)$ torus knot satisfies condition \eqref{BHconj}.
\end{proposition}
The proof will occupy the rest of this section.

\begin{lemma}
The ideal $J_1$ for the Brumfiel-Hilden algebra $H = H[\pi(p,2p+1)]$ is generated by the following relations:
\begin{eqnarray}
U_{p-1}(Q)X - U_{2p-1}(y) + U_{p-2}(Q)\label{yuo_a}\\
U_{p-1}(Q)X^{-1} - U_{2p-1}(y) + U_{p-2}(Q)\label{yuo_b}\\
2U_{p-2}(Q)y - U_{2p-2}(y)\label{yuo_c}\\
zU_{p-1}(Q)\label{yuo_d}
\end{eqnarray}
where $Q := T_2(y)x + 2yz = (2y^2-1)x + 2yz$.
\end{lemma}
\begin{proof}
Direct calculation similar to the one for $(p,p+1)$ torus knots.
\end{proof}
To prove condition \eqref{BHconj}, an argument similar to \eqref{yu_2_bp} shows that it is sufficient to prove
\begin{equation}\label{yuo_1}
U_{2p}(y) \in J_1
\end{equation}
By the covering argument of the previous section, we know that $l^2 \in H^+[X^{\pm 1}]$. This means that
\[
(T_{2p+1}(y) + U_{2p}(y)t)^2 \in H^+[X^{\pm 1}]
\]
or, equivalently, that 
\begin{equation}\label{yuo_2}
T_{2p+1}(y) U_{2p}(y) \in J_1
\end{equation}

We again will write $\equiv$ for congruence modulo the ideal $J_1$. We begin by rewriting (\ref{yuo_a})-(\ref{yuo_d}) and (\ref{yuo_2}) in a more concise form. Denote
\[
E := U_{p-1}(Q),\quad \quad N:= U_{2p-1}(y) - U_{p-2}(Q)
\]
\begin{lemma} 
We have
\begin{eqnarray}
EX &\equiv& N\label{yuo_3}\\
EX^{-1} &\equiv& N\label{yuo_4}\\
2yN &\equiv& U_{2p}(y)\label{yuo_5}\\
QE &\equiv& (2y^2-1)N \iff QN \equiv X^{-1}(2y^2-1)N\label{yuo_6}
\end{eqnarray}
\end{lemma}
\begin{proof}
First, (\ref{yuo_5}) follows from (\ref{yuo_c}). In particular, we have
\begin{eqnarray*}
(\ref{yuo_c}) \iff 2U_{p-2}(Q)y &\equiv& U_{2p-2}(y)\\
2y(U_{2p-1}(y) - N) &\equiv& U_{2p-2}(y)\\
2yN &\equiv& 2yU_{2p-1}(y) - U_{2p-2}(y) = U_{2p}(y)
\end{eqnarray*}

Second, we show that (\ref{yuo_6}) follows from (\ref{yuo_d}):
\begin{eqnarray*}
(\ref{yuo_d}) &\Rightarrow& 2yzU_{p-1}(Q) \equiv 0\\
(Q-(2y^2-1)x)U_{p-1}(Q) &\equiv& 0\\
QU_{p-1}(Q) &\equiv& (2y^2-1)xU_{p-1}(Q)\\
QE &\equiv& (2y^2-1)(X+X^{-1})E/2 = (2y^2-1)N
\end{eqnarray*}
\end{proof}
Thus, knowing (\ref{yuo_2})-(\ref{yuo_6}), we need to conclude (\ref{yuo_1}), i.e. that $U_{2p}(y) \equiv 0$. Recall that we assume $p$ to be odd. From (\ref{yuo_5}) we see that
\[
2yN \equiv U_{2p}(y) \quad \Rightarrow\quad  2yXE \equiv U_{2p}(y) \quad \Rightarrow\quad  2yX U_{p-1}(Q)N \equiv U_{2p}(y)N
\]
Since $p$ is odd, $U_{p-1}$ is an even polynomial, hence, by (\ref{yuo_3}) and (\ref{yuo_4}) we see $U_{p-1}(X^{-1}(2y^2-1))N \equiv U_{p-1}(2y^2-1)N$. If we formally set $y = \cos(\alpha)$, we see that $2y^2-1 = T_2(y) = \cos(2\alpha)$. From this, it follows that
\[
U_{p-1}(2y^2-1) = U_{p-1}(T_2(y)) = U_{p-1}(\cos(2\alpha)) = \frac{\sin(2 p \alpha)}{\sin (2 \alpha)}
\]
From this, we see that
\[
2yU_{p-1}(2y^2-1) = 2 \cos \alpha \frac{\sin 2p\alpha}{\sin 2\alpha} = \frac{\sin 2p\alpha}{\sin \alpha} = U_{2p-1}(y)
\]
Thus, we obtain
\begin{equation}\label{yuo_7}
X U_{2p-1}(y)N \equiv U_{2p}(y) N
\end{equation}
To proceed further, we need the following identity.
\begin{lemma}\label{lemma_yuo1}
For all $n \geq 0$, we have
\[
(1-y^2)U_n^2(y) + T_{n+1}^2(y) = 1
\]
\end{lemma}
\begin{proof}
Let $y = \cos \alpha$. Then 
\[
(1-\cos^2\alpha)\frac{\sin^2(n+1)\alpha} {\sin \alpha} + \cos^2(n+1)\alpha = \sin^2(n+1)\alpha + \cos^2(n+1)\alpha = 1
\]
\end{proof}

\begin{corollary}\label{cor_yuo1}
For all $p \geq 1$, we have
\begin{eqnarray*}
gcd(U_{2p}(y),yT_{2p+1}(y)) &=& 1\\
gcd(T_{2p}(y),yT_{2p+1}(y)) &=& 1
\end{eqnarray*}
\end{corollary}
\begin{proof}
To prove the first statement, let $n = 2p$ and define
\[
a(y) := (1-y^2)U_{2p}(y),\quad \quad b(y) = T_{2p+1}(y)/y
\]
Note that $b(y) \in \C[y]$ because $T_{2p+1}(y)$ is an odd polynomial. Then, by Lemma \ref{lemma_yuo1} we have 
\[
a(y)U_{2p}(y) + b(y)(yT_{2p+1}(y)) = 1
\]
The proof of the second statement is similar.
\end{proof}

Now, combining (\ref{yuo_7}) with (\ref{yuo_4}) and (\ref{yuo_5}), we get
\begin{eqnarray*}
U_{2p}(y)N &\equiv& X U_{2p+1}(y)N\\
yT_{2p+1}(y)N &\equiv& 0
\end{eqnarray*}
Hence, by Corollary \ref{cor_yuo1} we see
\[
N \equiv a(y)U_{2p-1}(y)NX
\]
Again using (\ref{yuo_4}) and (\ref{yuo_5}), $X^2N \equiv N$, which implies
\[
XN \equiv a(y)U_{2p-1}(y) N \quad \Rightarrow \quad N \equiv (a(y)U_{2p-1}(y))^2 N
\]
where $a(y) = (1-y^2)U_{2p}(y)$ as in the proof of Lemma \ref{lemma_yuo1}. We conclude
\begin{equation}\label{yuo_8}
\left[ \big(a(y)U_{2p-1}(y)\big)^2-1\right] N \equiv 0
\end{equation}
We then compute
\begin{eqnarray*}
(a(y)U_{2p-1}(y))^2-1 &=& (1-y^2)^2U^2_{2p}(y)U^2_{2p-1}(y) - 1\\
&=& \left[(1-y^2)U^2_{2p}(y)\right] \left[ (1-y^2)U^2_{2p-1}(y)\right] - 1\\
&=& (1-T_{2p+1}^2(y))(1-T_{2p}^2(y)) - 1\\
&=& T_{2p}^2(y) T^2_{2p+1}(y) - T_{2p}^2(y) - T^2_{2p+1}(y)
\end{eqnarray*}
Therefore, (\ref{yuo_8}) in combination with (\ref{yuo_2}) gives 
\begin{equation}\label{yuo_9}
T_{2p}^2(y)N \equiv 0
\end{equation}

By Corollary \ref{cor_yuo1}, $\mathrm{gcd}(T_{2p}(y),yT_{2p+1}(y)) = 1$. This shows that the polynomials $yT_{2p+1}(y)$ has no common roots with $T_{2p}(y)$, which means it also has no roots in common with $T^2_{2p}(y)$. We therefore have
\[
\mathrm{gcd}(T^2_{2p}(y),yT_{2p+1}(y)) = 1
\]
Combining this with (\ref{yuo_2}) and (\ref{yuo_9}) shows that
\[
N \equiv 0
\]
By (\ref{yuo_5}), we now conclude that 
\[
U_{2p}(y) \equiv 0
\]
This completes the proof of Proposition \ref{prop_p2pp1podd}.

\section{Pretzel knots}\label{sec_pretzel}\label{sec_pretzel}
In this section we will verify Conjecture \ref{bhconj} for some $(-2,3,2n+1)$ pretzel knots.

\subsection{Presentation and peripheral system}
It is shown in \cite[Prop. 2.1]{Nak13} (see also \cite[Sect. 4.1]{LT11}) that the knot group of a pretzel knot of type $(-2,3,2n+1)$ has the following presentation:
\[
\pi_1(K) = \langle a,b \mid b^n E = Fb^n\rangle
\]
where $E$ and $F$ are the following words in $F_2 = \langle a,b\rangle$:
\begin{equation}\label{yp1}
E := aba^{-1}b^{-1}a^{-1},\quad \quad F := a^{-1}b^{-1}abab^{-1}
\end{equation}
The peripheral system with this presentation is given by 
\begin{equation}\label{yp2}
m = a,\quad \quad l = a^{-2n+2}bab^nab^naba^{-2n-9}
\end{equation}
\begin{remark}
The above expression for the meridian and longitude have been found in \cite{Nak13}. Our notation differs from theirs: our generators $a^{\pm 1}$ and $b^{\pm 1}$ correspond to their $c$, $\bar c$ and $l$, $\bar l$.
\end{remark}

\subsection{The Brumfiel-Hilden algebra}
Recall (see the appendix) that the Brumfiel-Hilden algebra has the following presentation:
\[
H[\pi] = \left( R \oplus Rt\right) / (A + Bt)
\]
where $R = \C[X^{\pm 1},y,z]$, and 
\[
A + Bt := b^nE - Fb^n \in H[F_2]
\] 
To compute $A$ and $B$, we first observe that $F = E^\sigma b^{-1}$, where $\sigma: F_2 \to F_2$ is the involution of the free group defined by $\sigma(a) = a^{-1}$, $\sigma(b) = b^{-1}$, and $\sigma(ab) = \sigma(a)\sigma(b)$. This involution acts on $H[F_2] = R \oplus Rt$ by $X \mapsto X^{-1}$, $y \mapsto y$, $z \mapsto z$, and $t \mapsto -t$. We have
\[
E = aba^{-1}b^{-1}a^{-1} = ab(aba)^{-1} = (Xy + xt)(X^{-2}y + 2X^{-1}z - t)
\]
Write $E = E_0 + E_1t$; then a direct calculation shows
\begin{equation}\label{yp3}
E_0 = \alpha X + \beta,\quad \quad E_1 = \gamma X + \delta
\end{equation}
where we have used the elements
\begin{eqnarray}
\alpha &=& 1-4xyz - 2y^2 - 4z^2\label{yp4}\\
\beta &=& 2xy^2 + 2yz\notag\\
\gamma &=& 4x^2y+4xz-2y\notag\\
\delta &=& -2xy-2z\notag
\end{eqnarray}
Note that $E^\sigma = E_0^\sigma - E_1^\sigma t$, where $E_0^\sigma = \alpha X^{-1} + \beta$ and $E_1^\sigma = \gamma X^{-1}+\delta$. Hence,
\begin{eqnarray*}
A+Bt &=& b^nE - Fb^n \\
&=& b^nE - E^\sigma b^{n-1}\\
&=& (T_n(y) - U_{n-1}(y)t) (E_0 + E_1t) - (E_0^\sigma - E_1^\sigma t)(T_{n-1}(y) + U_{n-2}(y)t)
\end{eqnarray*}

By straightforward calculation, we then have
\begin{eqnarray}
A &=& T_n(y) E_0 - T_{n-1}(y) E_0^\sigma + U_{n-1}(y)\delta(E_0) + (U_{n-1}(y) + U_{n-2}(y)) E_1^\sigma (y^2-1)\label{yp5}\\
B &=& T_n(y)E_1 + T_{n-1}(y)E_1^\sigma + U_{n-1}(y) \delta(E_1) + (U_{n-1}(y) - U_{n-2}(y))E_0^\sigma\label{yp6}
\end{eqnarray}
where $\delta(E_0) = 2 \alpha z$ and $\delta(E_1) = 2 \gamma z$. It follows that
\begin{eqnarray}
A^\sigma &=& T_n(y) E_0^\sigma - T_{n-1}(y)E_0 + U_{n-1}(y)\delta(E_0) + (U_{n-1}(y) + U_{n-2}(y))E_1^\sigma (y^2-1)\label{yp7}\\
\delta(A) &=& \delta(E_0)(T_n(y) + T_{n-1}(y)) - \delta(E_1)(U_{n-1}(y) + U_{n-2}(y))(y^2-1)\label{yp8}\\
B^\sigma &=& T_n(y)E_1^\sigma + T_{n-1}(y)E_1 + U_{n-1}(y)\delta(E_1) + (U_n(y) - U_{n-2}(y))E_0\label{yp9}\\
\delta(B) &=& \delta(E_1)(T_n(y) - T_{n-1}(y)) - \delta(E_0)(U_{n-1}(y) - U_{n-2}(y))\label{yp10}
\end{eqnarray}
These computations combined with the computations in Appendix \ref{sec_app} show the following.
\begin{lemma}
The ideal $J$ is generated by 
\begin{equation}\label{yp_defofJ}
 J = \langle A, B, A^\sigma + \delta(B), \delta(A) + B^\sigma(y^2-1)\rangle
\end{equation}
where the elements $A$, $B$, etc.\ are given in equations \eqref{yp5} through \eqref{yp10}.
\end{lemma}
\begin{proof}
 This follows from Proposition \ref{yuri_prop2}.
\end{proof}
\begin{remark}
 The symmetry condition \eqref{cond_star}, which was true in the case of torus knots, does \emph{not} hold for pretzel knots (cf.\ \eqref{yp6}). However, computer experiments suggest that (for small $n$), 
 \[
  J = \langle A,B,A^\sigma, B^\sigma,\delta(A), \delta(B)\rangle
 \]
This seems hard to prove in general.
\end{remark}

\subsection{Computing the longitude}
By \eqref{yp2}, it suffices to prove that 
\[
 \bar l := bab^nab^nab \in H_0[\pi]
\]
To compute this element we use the (anti-)involution $\gamma: F_2 \to F_2$ given by $a \mapsto a$ and $b \mapsto b$ (so that $ab \mapsto ba$). It acts on $H[\pi]$ by 
\[
 X \mapsto X,\quad y \mapsto y,\quad z \mapsto z,\quad t \mapsto t
\]
Thus
\begin{equation}\label{yp11}
 \bar l = bab^nab^nab = (bab^n)a\gamma(bab^n) = (C+Dt)X(C+Dt)
\end{equation}
where $bab^n = C+Dt$ with
\begin{eqnarray}
 C &:=& (Xy+2z)T_n(y) + X^{-1}(y^2-1)U_{n-1}(y)\label{yp12}\\
 D &:=& (Xy+2z)U_{n-1}(y) + X^{-1}T_n(y)\label{yp13}
\end{eqnarray}
It follows that $\bar l = \bar l_0 + \bar l_1 t$, where 
\begin{equation}\label{yp14}
 \bar l_1 = C D^\sigma X + DC^\sigma X^{-1} + 2 D D^\sigma z
\end{equation}
with
\begin{eqnarray}
 C^\sigma &=& (X^{-1}y+2z)T_n(y) + X(y^2-1)U_{n-1}(y)\label{yp15}\\
 D^\sigma &=& (X^{-1}y+2z)U_{n-1}(y) + XT_n(y)
\end{eqnarray}

A computer calculation with \texttt{Maple} shows that for all $n \leq 20$ the element $\bar l_1$ defined above belongs to the ideal $J$ defined in \eqref{yp_defofJ}. This implies
\begin{theorem}
 The $(-2,3,2n+1)$ pretzel knots satisfy Conjecture \ref{bhconj}, at least for $n \leq 20$.
\end{theorem}
\begin{remark}
 With enough effort it should be possible to verify the inclusion $\bar l_1 \in J$ for all $n$.
\end{remark}

\section{Two-bridge knots}\label{sec_twobridge}
In \cite{BS16} we confirmed Conjecture \ref{conj_bs14} for 2-bridge knots using explicit computations from \cite{BH95} and a $\c[X^{\pm 1},Y^{\pm 1}]\rtimes \Z_2$-submodule of $H[\delta^{-1}]$ defined as
\[
 M := H^+[X^{\pm 1}] + H^+[X^{\pm 1}]Q\delta^{-1}
\]
(See \cite[Eq.\ (3.11)]{BS16}.) In this section we show that the module $M$ is equal to the module $N$ from \eqref{eq_defN}. In particular, this shows that $M$ has a definition that does not depend on the polynomial $Q$, which was a specific polynomial used in the computations of \cite{BH95}. We will adapt the notation of \cite{BS16}.

\begin{proposition}
 If $K$ is a two-bridge knot, then $N = H^+[X^{\pm 1}] + H^+[X^{\pm 1}]Q\delta^{-1}$. In particular, the Brumfiel-Hilden condition \eqref{eq_bhintro} holds for $K$.
\end{proposition}
\begin{proof}
 First, by Lemma 3.5 in \cite{BS16}, it is clear that $H^+[X^{\pm 1},Y^{\pm 1}] = H^+[X{\pm 1}]$ in $H[\pi]$. By \cite[Proof of Thm.\ 3.7]{BS16}, we know that 
 \begin{equation}\label{y2b1}
  Y = fQ + g \delta - 1
 \end{equation}
where we have written
\begin{eqnarray*}
 f &=& 2X^{-s}(L+NJ) \in H^+[X^{\pm 1}]\\
 g &=& X^{-s}\left( 2N^2J\delta + 2LM + A(X)\right) \in H^+[X^{\pm 1}]\\
 s &=& 4\sum_{n=1}^d e_n\\
 A(X) &=& \left\{ 
 \begin{array}{ll} 
 X+X^3+\cdots+X^{s-1} & \textrm{if } s \not= 0\\
  0,&\textrm{if } s=0                 
 \end{array} \right.
\end{eqnarray*}
It follows from \eqref{y2b1} that $(y+1)\delta^{-1} = g + fQ\delta^{-1}$, which implies $N \subset M$.

To prove the inclusion $M \subset N$, we note that \cite[pg.\ 119]{BH95} shows that
\begin{eqnarray*}
 L &=& 1 + 2F^2J - 2G^2IJ\\
 N&=& 2DG + 2EF
\end{eqnarray*}
This shows that $fQ = 2X^{-s}(L + NJ)Q$, which implies
\begin{eqnarray*}
 fQ &=& 2X^{-s}(1+2F^2J-2G^2IJ+2DGJ+2EFJ)Q\\
 &=& 2X^{-s}[1+2(F^2+DG+EF)J]Q\\
 &=& 2X^{-s}[1+2(F^2+DG+EF)\delta^2]Q
\end{eqnarray*}
because $IQ=0$ and $JQ=\delta^2Q$ in $H$. Hence 
\[
 (Y+1)\delta^{-1} = (fQ+g\delta)\delta^{-1} = fQ\delta^{-1}+g \in N
\]
This implies 
\[
 fQ\delta^{-1} = (Y+q)\delta^{-1}-g \in N
\]
which then implies
\[
 2X^{-s}[1+2(F^2+DG+EF)\delta^2]Q\delta^{-1} = 2X^{-s}[Q\delta + 2(F^2+DG+EF)\delta Q] \in N
\]
This implies $2X^{-s}Q\delta^{-1} \in N$, which implies $Q\delta^{-1} \in N$, which finally implies $M \subset N$.

Finally, the last statement follows \cite[Thm.~3.9]{BS16}, which proves the conditions in Proposition \ref{prop_3conditions} for the module $M$.
\end{proof}

\section{Further remarks}\label{sec_furtherremarks}
In this section we provide some further remarks about the Brumfiel-Hilden condition \eqref{bhcond}. First, we propose a generalization from $\SL_2(\C)$ to $\SL_n(\C)$ (although we will leave the problem of relating this to higher rank DAHAs to later work). Second, it is natural to ask whether there is a condition on the $A$-polynomial of a knot $K$ that implies the Brumfiel-Hilden condition for $K$. We show that there is such a condition on $A$, but that it does not hold for the figure eight knot or for some torus knots. This makes it seem less likely that the Brumfiel-Hilden condition can be proved using properties of the $A$-polynomial.

\subsection{Higher rank generalization}


%

Given a group $\pi$, let $\rep_n(\pi) := \Hom(\pi,\SL_n)$ be the variety of representations of $\pi$ into $\SL_n(\c)$ (which are not considered up to isomorphism). We also define
\begin{equation}\label{hndef}
H_n[\pi] := \Gamma(\rep_n(\pi),M_n(\C))^{\GL_n},\quad \quad H^+_n[\pi] := \Gamma(\rep_n(\pi),\c)^{\GL_n}
\end{equation}
Here if $X$ is a space and $V$ a vector space, we have written $\Gamma(X,V)$ for $V$-valued functions on $X$. If $G$ acts on $X$ and $V$, then $\Gamma(X,V)^G$ is the space of $G$-equivarient $V$-valued functions. The action of $\GL_n$ on the space $M_c(\c)$ of $n\times n$ matrices is by conjugation, and the action of $\GL_n$ on $\c$ is trivial. 

By definition, $H^+_n[\pi]$ is the ring of functions on the $\SL_n$ character variety of $\pi$. We remark that one easy source of equivariant sections in $H_n[\pi]$ are evaluations at elements of $\pi$: given $g \in \pi$, define
\[
\ev_g: \rep(\pi) \to M_n(\c),\quad\quad \rho \mapsto \rho(g)
\]
Similarly, an element $g \in \pi$ produces a function $\rho \mapsto \Tr(\rho(g))$ in $H^+_n[\pi]$.

We now give a statement which implies Conjecture \ref{conj_bs14} when $n=2$ and $q=-1$. However, we remark that we have no evidence for this statement other than $n=2$. We will write 
\[ 
X := \ev_m \in H[\pi],\quad Y:= \ev_l \in H[\pi]
\] 
where $m$ and $l$ are the meridian and longitude of the knot, and we will use the fact that $H_n[\pi]$ is an algebra (where the multiplication is ``pointwise'' and comes from matrix multiplication).

\begin{conjecture}
We (optimistically) believe the following inclusion holds:
\begin{equation}\label{wildstatement}
H^+_n[\pi][X^{\pm 1},Y^{\pm 1}] \subset H^+_n[\pi][X^{\pm 1}]
\end{equation}
\end{conjecture}

The left hand side of this (conjectural at best) inclusion is the subalgebra of $H_n[\pi]$ generated by $H^+[\pi]$ and the elements $X^{\pm 1}$ and $Y^{\pm 1}$, and similarly for the right hand side. (The reverse inclusion is obvious.)

\begin{remark}
For $n=2$, (\ref{wildstatement}) appeared as a conjecture in the last sentence of \cite[pg. 122]{BH95}.
\end{remark}

\subsection{The $A$-polynomial and the Brumfiel-Hilden condition}
Recall that for a knot $K \subset S^3$ with $\pi = \pi_1(S^3\setminus K)$ we define the algebra map 
\[
 \alpha: \C[m^{\pm 1},l^{\pm 1}] \to H[\pi]
\]


We now define two ideals in $\C[m^\pmone,l^\pmone]$:
\[
J := \Ker(\alpha) \subset \C[m^{\pm 1},l^{\pm 1}],\quad \quad \langle J,m-m^{-1}\rangle\subset \C[m^{\pm 1},l^{\pm 1}]
\]
Consider the following condition:
\begin{equation}\label{eq_cond4}
 l-l^{-1} \in \langle J,m-m^{-1}\rangle \subset \C[m^{\pm 1},l^{\pm 1}]
\end{equation}

\begin{lemma}
 Condition (\ref{eq_cond4}) implies the conditions in Proposition \ref{prop_3conditions}.
\end{lemma}
\begin{proof}
 If $l - l^{-1} \in \langle J,m-m^{-1}\rangle$, then we can apply $\alpha$ to obtain 
 \[ Y - Y^{-1} \in \alpha(\C[m^\pmone,l^\pmone])(X-X^{-1}) \subset H\delta \]
\end{proof}

\begin{remark}
 Condition (\ref{eq_cond4}) is equivalent to 
 \[
  l - l^{-1} \equiv f(m,l)(m-m^{-1}) \,\, (\textrm{mod } J)
 \]
for some $f(m,l)$.
\end{remark}
By \cite[Prop 10.5]{BH95}, the ideal $J$ has primary decomposition 
\[
J = \pp_1 \cap \pp_2 \cap \cdots \cap \pp_n
\]
where $\pp_i$ are the primary components belonging to the prime ideals of dimension 1. These primary ideals are generated by powers of irreducible polynomials in $\C[m^\pmone,l^\pmone]$, i.e. 
\[
\pp_i = \langle p_i^{m_i}\rangle
\]
and the product of the $p_i$'s is the $A$-polynomial:
\[
A_K(m,l) = \prod_{i=1}^k p_i
\]
It follows that 
\begin{equation}\label{eq_y5}
J \subset \langle A_K(m,l)\rangle
\end{equation}
The \emph{BH conjecture} is that the containment \ref{eq_y5} is actually an equality.

We will need a weaker version:
\begin{conjecture}[Weak BH conjecture]\label{conj_bhweak}
$\langle J,m-m^{-1}\rangle = \langle A_K(m,l),m-m^{-1}\rangle$
\end{conjecture}
We next introduce the following condition:
\begin{equation}\label{eq_y6}
l - l^{-1} \in \langle A_K(m,l),m-m^{-1}\rangle
\end{equation}

\begin{lemma}\label{lemma_eqcond}
Condition \eqref{eq_y6} is equivalent to
\begin{equation}\label{eq_y6p}
A_K(\pm 1, l) \textrm{ divides } l^2 - 1 \in \C[l^\pmone]
\end{equation}
\end{lemma}
\begin{proof}
(We refer to Cooper-Long's 1998 paper \cite{CL98} for results on the $A$-polynomial.
By \cite[Thm. 3.5]{CL98}, $A_K(m,l)$ can be normalized so that $A(m,l) = P(m^2,l) \in \C[m^2,l]$. Write
\begin{equation}\label{eq_msq}
P(m^2,l) = \sum_{i=0}^n P_i(l)(m^2-1)^i
\end{equation}
so that $P_0(l) = P(1,l) = A_K(\pm 1,l)$. Then
\begin{eqnarray*}
\langle A(m,l),m-m^{-1}\rangle &=& \langle P(m^2,l), m^2-1\rangle\\
&=& \langle P(1,l), m^2-1\rangle\\
&=& \langle A(\pm 1, l), m^2-1\rangle
\end{eqnarray*}
Thus (\ref{eq_y6}) holds iff $(l^2-1) \in \langle A(\pm 1,l),m^2-1\rangle$, which is equivalent to $l^2-1 = a(l,m)A(\pm 1,l) + b(l,m)(m^2-1)$. Finally, this equality can only hold when $b(l,m) = 0$, which implies equation (\ref{eq_y6p}). 

Conversely, the expansion (\ref{eq_msq}) shows that (\ref{eq_y6p}) implies (\ref{eq_y6}).
\end{proof}

\begin{corollary}\label{cor_adiv}
Condition (\ref{eq_y6p}) combined with the weak BH Conjecture \ref{conj_bhweak} imply \cite[Conj. 1]{BS16} at $q=-1$. We also have the following implication:
\[
l - l^{-1} \in \langle J,\, m-m^{- 1}\rangle \quad \Rightarrow \quad A_K(\pm 1,l) \mid (l^2-1)
\]
\end{corollary}

\subsubsection{Examples}
We now give some examples to show that condition (\ref{eq_y6p}) does not hold in general.

\begin{example}[$(2,2p+1)$ torus knots]
\[
A_K(m,l) = (l-1)(1+lm^{2(2p+1)}),\quad A_K(\pm 1,l) = l^2-1
\]
Therefore, condition (\ref{eq_y6p}) holds. One can compute
\[
(Y-1)(Y+X^{-2(2p+1)}) = 0 \in H[\pi]
\]
A computation then implies
\begin{eqnarray*}
(Y-Y^{-1})\delta^{-1} &=& X(1 + X^2 + \cdots + X^{4p})(1-Y)\\
&=& (1/2)\left[ \Tr(X^{2p+1} - \Tr(YX^{2p+1})) + \sum_{k=1}^p (\Tr(X^{2k-1} -  \Tr(YX^{2k-1})\right]
\end{eqnarray*}

\end{example}
\begin{example}[The $(p,q)$ torus knots with $p,q > 2$:]
In this case,
\begin{eqnarray*}
A_K(m,l) &=& (l-1)(-1+l^2m^{2pq})\\ 
A_k(\pm 1,l) &=& (l-1)(l^2-1) 
\end{eqnarray*}
Since $A_K(\pm 1,l)$ does not divide $l^2-1$, we see that $l-l^{-1} \notin \langle J,m-m^{-1}\rangle$.

\end{example}

\begin{proposition}\label{prop_notsubmod}
Assume that the (equivalent) conditions of Proposition \ref{prop_3conditions} hold for a knot $K$ but condition \ref{eq_y6p} fails. Then
\begin{enumerate}
\item the map $\alpha$ is \emph{not} surjective,
\item $\mathrm{Im}(\alpha)$ is \emph{not} preserved by $U  = \delta^{-1}(1+sY)$.
\end{enumerate}
\end{proposition}
\begin{proof}
Assume that $\alpha$ is surjective. Then the first condition of Proposition \ref{prop_3conditions} implies $Y-Y^{-1} = F(X^\pmone,Y^\pmone)\delta$ for some $F$. This implies 
\begin{eqnarray*}
\alpha\left(l-l^{-1}-F(m^\pmone,l^\pmone)(m-m^{-1})\right) &=& 0,\quad \Rightarrow\\
l-l^{-1}-F(m^\pmone,l^\pmone)(m-m^{-1}) &\in& J,\quad \Rightarrow\\
l-l^{-1} \in \langle J,m-m^{-1}\rangle &\subset & \langle A_K(m,l,m-m^{-1}) 
\end{eqnarray*}
However, by Lemma \ref{lemma_eqcond}, this is equivalent to $A_K(\pm 1,l) \mid (l^2-1)$, which is a contradiction.

To prove the second claim, we argue by contradiction and assume that $U[\mathrm{Im}(\alpha)] \subset \mathrm{Im}(\alpha)$. Since we have already assumed that $UM \subset M$ and $\mathrm{Im}(\alpha) \subset M$, we see that $\mathrm{Im}(\alpha)$ is a submodule of $M$. Since $1 \in \mathrm{Im}(\alpha)$, we have that
\[
(YU - UY^{-1})1 \in \mathrm{Im}(\alpha)
\]
But $YU - UY^{-1} = \delta^{-1}(Y+s-Y^{-1}-s) = \delta^{-1}(Y-Y^{-1})$, and this combined with our assumptions shows that $(Y-Y^{-1})\delta^{-1} \in \mathrm{Im}(\alpha)$, which leads to a contradiction as in the first claim.
\end{proof}
\begin{example}[The figure eight]
One may compute
\[
(Y-Y^{-1})\delta^{-1} = -(I+IJ)\Tr(X) \in H[\pi]
\]
Then the proof of Proposition \ref{prop_notsubmod} implies that $I+IJ \notin \mathrm{Im}(\alpha)$ since condition \eqref{eq_y6p} fails to hold in this case.
\end{example}

\begin{remark}
By \cite[Theorem 3.6]{CL98}, $A(\pm 1,l)$ always divides $(l^2-1)^N$ for some $N$ under a certain (probably unnecessary) technical condition. We therefore have property (\ref{eq_y6}) in a weaker form:
\[
(l-l^{-1})^N \in \langle A_K(m,l),m-m^{\pm 1}\rangle
\]
for some $N \geq 1$. If the weak BH conjecture is true, this implies
\[
(Y-Y^{-1})^N \in H\delta
\]
\end{remark}

\section{Appendix: The Brumfiel-Hilden algebra}\label{sec_app}
\subsection{The BH algebra of a free group on two generators}
Let $ F := \langle a,b\rangle$ be the free group on two generators $a$ and $b$.
We begin by recalling the presentation of the BH algebra $ H[F] $ of $F$ given in \cite{BH95}. 
We will use the following notation (cf. \cite{BH95}): for any $ g \in F $, we write
$ g^+ := (g+g^{-1})/2 $ and $ g^- := (g-g^{-1})/2 $ in $ H[F] $; also, we set
\begin{eqnarray*}
 \lvert a \rvert & :=& a^- = (a-a^{-1})/2\\
 \lvert b \rvert & :=& b^- = (b-b^{-1})/2 \\
 \lvert ab \rvert & :=& (a^-b^-)^- = ab - b^{+} a - a^{+}b - (ab)^{+} + 2 a^+ b^+ \\
 x &:=& a^+ \\
 y &:=& b^+\\
 z &:=& (a^-b^-)^+ = (ab)^+ - a^+ b^+
\end{eqnarray*}
\begin{theorem}[{\cite[Prop 3.9]{BH95}}]
\la{BHT1}
 For $F = \langle a,b\rangle$, we have $\, H[F] = H^+[F] \oplus H^-[F]\,$, where
 \begin{enumerate}
  \item $H^+[F] = k[x,y,z]$ is a free polynomial ring
  \item $H^-[F] = H^+\lvert a \rvert \oplus H^+\lvert b \rvert \oplus H^+\lv{ab}$ is a free $H^+[F]$-module 
\end{enumerate}

The multiplication table in $H$ is given by
  \[
   \begin{array}{c|ccc} \, & \lvert a \rvert & \lvert b \rvert & \lvert ab \rvert \\ \hline
    \lvert a \rvert & x^2-1 & z + \lvert ab \rvert & -z\lv a + (x^2-1) \lv b\\
    \lv b & z - \lv{ab}& y^2-1 & -(y^2-1)\lv a + z \lv b \\
    \lv {ab} & z\lv a - (x^2-1)\lv b & (y^2-1)\lv a - z \lv b & z^2 - (x^2-1)(y^2-1)
   \end{array}
  \]
\end{theorem}
%
%

We will construct a different presentation of $H[F]$ in terms of Ore extensions. First,
we recall basic definitions.

\subsection{Ore extensions}
Let $ k $ be a field. Given an associative $k$-algebra $R$ and an endomorphism $\sigma:R \to R$, a (left) $\sigma$-derivation $\delta:R \to R$ is a linear map satisfying the rule
\[
 \delta(ab) = \sigma(a)\delta(b) + \delta(a)b
\]
Note that this rule implies $\delta(a) = 0$ for $a \in k$. 
A typical example of a $\sigma$-derivation is given by a twisted inner derivation 
$ ad_\sigma(a): R \to R $ defined by $ \ad_\sigma(a)[x] := ax - \sigma(x)a $, where
$ a \in R $.

\begin{definition}
 Given a commutative $k$-algebra $R$, an endomorphism $\sigma:R \to R$, and a 
$\sigma$-derivation $\delta$, we define
\begin{equation*}
  R[t;\sigma,\delta] := \frac{R\langle t\rangle}{(ta = \sigma(a) t + \delta(a))}\ .
 \end{equation*}
The algebra $ R[t;\sigma,\delta]$ is called the \emph{Ore extension} of $R$ with respect to $(\sigma,\,\delta)$.
\end{definition}
The following facts are standard and easy to prove (see, e.g., \cite[1.2.4]{MR01}).
\begin{lemma}
\la{oreex}
\begin{enumerate}
 \item If $R$ is an integral domain and $ \sigma $ is injective, then  
$ R[t;\sigma,\delta] $ is a (noncommutative) domain, which is free as a left 
and right module over $R$.
\item If $R$ is left (right) Noetherian, then $R[t;\sigma,\delta]$ is also left (right) Noetherian.
\end{enumerate}
\end{lemma}
\subsection{The BH algebra as an Ore extension} 
Let $\, R := k[X^{\pm 1}, y,z]\,$, and let $ \sigma \in \Aut(R) $ be an automorphism of $ R$ 
defined on generators by
\begin{equation}
\la{sigma}
  X \mapsto X^{-1},\quad X^{-1} \mapsto X,\quad y \mapsto y,\quad z \mapsto z
\end{equation}
To define $\delta$, we set
\[
 \delta(X) = 2z,\quad \delta(X^{-1}) = -2z,\quad \delta(y) = \delta(z) = 0\ .
\]
\begin{lemma}
\la{delta}
 $\delta$ extends to a (unique) $\sigma$-derivation of $R$ given by 
\begin{equation}
\la{form}
\delta(X^k) = 2z \,U_{k-1}(x)\ , \quad \forall\, k \in \Z\ ,
\end{equation}
where $ U_{k-1}(x) $ is the $(k-1)$-th Chebyshev polynomial of $\, x = (X+X^{-1})/2 $. This derivation satisfies the relations
\begin{equation}
\la{reld}
\sigma\,\delta = -\delta\,\sigma = \delta\ ,\quad \delta^2 = 0 \ . 
\end{equation}
\end{lemma}
\begin{proof}
First, check that $\delta(X\cdot X^{-1}) = \delta(1) = 0$. Indeed, 
 \[\delta(X\cdot X^{-1}) = \sigma(X)\delta(X^{-1}) + \delta(X)X^{-1} = X^{-1}(-2z)+2z X^{-1} = 0 
 \]
Now, formula \eqref{form} follows easily by induction in $k$, while the relations 
\eqref{reld} follow \eqref{form}.
\end{proof}
Using the above $\sigma$ and $\delta$ on $R$, we define
\begin{equation}
\label{eq_orepres}
 R[t;\sigma,\delta] = \frac{k[X^{\pm 1},y,z]\langle t 
\rangle}{(ty=yt,\,\, tz = zt,\,\, tX = X^{-1}t + 2z,\,\,tX^{-1} = Xt - 2z)}
\end{equation}
Note that, by Theorem~\ref{BHT1}, the $\sigma$-invariant subalgebra 
$ R^\sigma = k[x,y,z]$ of $R$ is isomorphic to $ H^+[F]$. 
The next proposition shows that this isomorphism can be extended to
the entire algebra $ H[F] $.
\begin{proposition}
\la{presBH}
The assignment $\, X^{\pm 1} \mapsto a^{\pm 1}$, 
$\,y\mapsto y$, $\, z \mapsto z$, $\,t \mapsto\lv b\,$
extends to a surjective algebra homomorphism $\Psi: R[t;\sigma,\delta] \onto H[F]$,
with kernel generated by $ t^2-y^2+1$. Thus, we have an isomorphism of algebras
\[
 H[F] \cong R[t;\sigma,\delta]/(t^2-y^2+1)
\]
that restricts to $ H^+[F] \cong R^\sigma = k[x,y,z] $.
\end{proposition}
\begin{proof}
The map $\Psi$ is well defined because
 \[
  \Psi(tX-X^{-1}t-2z) = \lv b (x+\lv a) - (x - \lv a)\lv b - 2z = \lv b \lv a + \lv a \lv b - 2z = z - \lv{ab} + z + \lv{ab} - 2z = 0\ .
 \]
Similarly, $\Psi(tX^{-1} - Xt + 2z) = 0$, and 
\[
 \Psi(1) = \Psi(XX^{-1}) = (x+\lv a)(x - \lv a) = x^2 - {\lv a}^2 = 1
\]
The surjectivity of $\Psi$ follows from the calculation
\begin{eqnarray*}
 \Psi([X,t]) &=& \Psi(Xt) - \Psi(tX) = (x+\lv a)\lv b - \lv b (x+\lv a)\\
 &=& \lv a \lv b - \lv b \lv a = z + \lv {ab} - (z - \lv {ab}) = 2\lv {ab}\ .
\end{eqnarray*}
Finally, we check that
\[
 \Psi(t^2-y^2+1) = \lv{b}^2 - y^2 + 1 = y^2 - 1 - y^2 + 1 = 0
\]
and it is easy to see that $ t^2-y^2+1 $ generates the kernel of $ \Psi $.
\end{proof}
Thus, we have the following commutative diagram
\begin{diagram}[small]
 \frac{R[t;\sigma,\delta]}{t^2-y^2+1} & \rTo^{\cong} & H[F]\\
 \uInto & & \uInto \\
 R & \rTo^{\cong} & H^+[a^{\pm 1}]\\
 \uInto & & \uInto\\
 R^\sigma & \rTo^{\cong} & H^+
\end{diagram}
where the horizontal arrows are algebra isomorphisms and the vertical ones are inclusions.

\begin{corollary}\label{cor_hpres}
The algebra $H[F]$ is a free quadratic extension of $ R = H^+[a^{\pm 1}]$ with basis 
$\{1,t\}$:
 \[
  H[F] = R \oplus Rt
 \]
The multiplication in $ H[F] $ is determined by the relation $ t^2 = y^2-1$.
\end{corollary}
\begin{proof}
This follows from the isomorphism of Proposition~\ref{presBH} and Lemma~\ref{oreex}$(1)$.
\end{proof}
We list several natural involutions on $H[F]$ which are induced from involutions on the 
free group $ F $:
\begin{enumerate}
 \item Inverse anti-involution: $*: H[F] \to H[F]$, defined by $a \mapsto a^{-1}$ and $b \mapsto b^{-1}$.
 \item Canonical anti-involution: $\gamma: H[F] \to H[F]$, defined by $a \mapsto a$, $b \mapsto b$, and $\gamma(ab) = ba$.
 \item Involution ``switching $a$ and $b$:'' $\xi: H[F] \to H[F]$, $a \mapsto b$, $b \mapsto a$.
 \item Involution inverting $a$ and $b$: $\sigma:H[F] \to H[F]$, $a \mapsto a^{-1}$, $b \mapsto b^{-1}$.
\end{enumerate}
These involutions can be expressed in terms of their action on the generators 
$X^{\pm 1}$, $y$, $z$ and $t$. For example, the involution $ \sigma $ extends 
the eponymous involution on $ R \,$: it acts on $X^{\pm 1}$, $y$ and $z$ by 
\eqref{sigma}, while $ \sigma(t) = - t $.

\subsection{One-relator groups}
We now use the results of the previous section to give a presentation of the Brumfiel-Hilden algebra for
a two generator groups with one relator. Thus, we consider $\pi = F\langle a,b \rangle/(w_1 w_2^{-1})$, where $ w_1 $ and $w_2 $ are two words in $ F\langle a,b \rangle $. By Corollary 
\ref{cor_hpres}, $\, H[F] = R \oplus Rt\,$, with multiplication defined by  $ t^2 = y^2-1\,$.
Hence, 
\[
w_1 - w_2 = A + Bt \in H[F]
\]
for some polynomials $\,A, B \in R = k[X^{\pm 1},y,z]$. We will write $(w_1-w_2) = (A+Bt)$ for the two-sided ideal in $ H[F] $ generated by $ w_1 - w_2  $. 

By [BH95, Prop.~1.4] and Proposition~\ref{presBH} above, we can now identify
\begin{equation}
\la{BHpi}
H[\pi] \cong  R[t;\sigma,\delta]/(A+Bt, t^2 - y^2+1) \cong [R \oplus Rt]/(A+Bt)\ ,
\end{equation}
and think of elements of $ H[\pi] $ as elements of $ R \oplus Rt $ modulo the ideal
$(A+Bt)$. Our goal is to give a criterion when an element $ [r_0 +r_1 t] \in  H[\pi] $ 
belongs to the commutative subalgebra $ H_0[\pi] $ generated by $ H^+[\pi] $ and $ a^{\pm 1} \in H[\pi] $. First, observe that, with identification \eqref{BHpi}, we have
\begin{equation}
\la{BHpi0}
H_0[\pi] \cong R/R\,\cap\,(A+Bt)\ .
\end{equation}
To characterize the elements of $ H_0[\pi] $ we define
$$
J := \{r \in R\, :\, r_0 +rt \in (A+Bt)\ \mbox{for some}\, r_0 \in R\}\ .
$$
Note that $\, r \in J \,$ $ \Leftrightarrow $ $\,rt \equiv R\,$ modulo $\,(A+Bt)\,$.
Hence $ J $ is an ideal of $ R $ such that $\, R\,\cap\,(A+Bt) \subseteq J \,$.
(The last inclusion follows from the fact that for any $\, r \in R\,\cap\,(A+Bt) $,
$ rt \in (A+Bt) $, because $ (A+Bt) $ is an ideal of $ H[F] $, so $ rt \equiv 0 $ 
modulo $\,(A+Bt)\,$.) Moreover, it is immediate from \eqref{BHpi} that
\begin{equation}
\la{incl}
[r_0 +r_1 t] \in  H_0[\pi] \quad \Leftrightarrow \quad r \in J\ .  
\end{equation}
To make \eqref{incl} an effective criterion we need to give generators of $ J $ 
as an ideal in $ R $. First, by definition, $ B \in J $. On the other hand, we have
\begin{eqnarray*}
 (A+Bt)t &=& B(y^2-1) + At\\
 t(A+Bt) &=& tA + tBt = A^\sigma t + \delta(A) + B^\sigma(y^2-1) + \delta(B)t \\ 
 &=& \delta(A) + B^\sigma(y^2-1) + (A^\sigma + \delta(B))t\\
 t(A+Bt)t &=& (A^\sigma + \delta(B))(y^2-1) + (\delta(A) + B^\sigma(y^2-1) t
\end{eqnarray*}
where $\, A^\sigma := \sigma(A)$, etc. Hence $ J $ also contains the elements
$\,A\,$, $\, A^\sigma + \delta(B)$, $\,\delta(A) + B^\sigma(y^2-1)$, 
$\, A^\sigma - \delta(B^\sigma)$ and $\, \delta(A^\sigma) - B^\sigma(y^2-1)\,$. 
Now, by \eqref{reld}, the last two elements coincide with the previous and hence
are redundant as generators; on the other hand, it is easy to see that 
the rest do actually generate $J\,$: i.e.,
\begin{equation}
\label{eq_j1}
J = \langle A,\,B,\, A^\sigma + \delta(B),\,\delta(A) + B^\sigma(y^2-1) \rangle
\end{equation}

Let us assume that $ \pi $ admits a palindromic presentation, i.e.
the defining relation $ w_1 = w_2 $ of $ \pi $ is such that
the elements $ w_1 $ and $ w_2 $ are both palindromic as words in $a$ and $b$. 
Then $ \gamma(w_1) = w_1 $ and $ \gamma(w_2) = w_2 $ and hence 
$ \gamma(w_1 - w_2) = w_1 - w_2 $, where $ \gamma $ is the canonical 
anti-involution fixing $a$ and $b$.  The last condition implies that 
$ tB = Bt $ in $ H[F] $ or equivalently,
\begin{equation}
\label{cond_star}
 B^\sigma = B \ . 
\end{equation}
In this case, the ideal $ J $ simplifies as follows
\begin{equation}
\label{eq_j1}
J = \langle A,\,B,\, A^\sigma,\,\delta(A)\rangle\ .
\end{equation}
We summarize the above observations in the following Proposition which is the main tool used in Sections \ref{sec_torus} and \ref{sec_pretzel}.
\begin{proposition}
\label{yuri_prop2}
Let $\, \pi = \langle a, b\, :\, w_1 w_2^{-1} \rangle \,$ be a group with two generators
and one relator. Let $ w_1 - w_2 = A + Bt \in H[F] $. Then
\begin{enumerate}
\item[(1)]
$\,H[\pi] \cong  R[t;\sigma,\delta]/(A+Bt, t^2 - y^2+1) \quad\mbox{and} 
\quad H_0[\pi] \cong R/(R\,\cap\,(A+Bt))$.
\item[(2)]
An element $ [r_0 + r t] \in H[\pi] $ belongs to $ H_0[\pi] $ if and only if
$$
r \in J = \langle A,\,B,\, A^\sigma + \delta(B),\,\delta(A) + B^\sigma(y^2-1) \rangle\ .
$$
\item[(3)] If the presentation of $ \pi $ has a palindromic presentation, then
$ J $ in $(2)$ is given by
$$
J = \langle A,\,B,\, A^\sigma,\,\delta(A)\rangle\ .
$$
\end{enumerate}
\end{proposition}

\bibliography{affine_cubic_surfaces}{}

\newcommand{\etalchar}[1]{$^{#1}$}
\providecommand{\bysame}{\leavevmode\hbox to3em{\hrulefill}\thinspace}
\providecommand{\MR}{\relax\ifhmode\unskip\space\fi MR }
\providecommand{\MRhref}[2]{%
  \href{http://www.ams.org/mathscinet-getitem?mr=#1}{#2}
}
\providecommand{\href}[2]{#2}
\begin{thebibliography}{CCG{\etalchar{+}}94}

\bibitem[BH95]{BH95}
G.~W. Brumfiel and H.~M. Hilden, \emph{{${\rm SL}(2)$} representations of
  finitely presented groups}, Contemporary Mathematics, vol. 187, American
  Mathematical Society, Providence, RI, 1995. \MR{1339764 (96g:20004)}

\bibitem[BP00]{BP00}
Doug Bullock and J{\'o}zef~H. Przytycki, \emph{Multiplicative structure of
  {K}auffman bracket skein module quantizations}, Proc. Amer. Math. Soc.
  \textbf{128} (2000), no.~3, 923--931. \MR{1625701 (2000e:57007)}

\bibitem[BS16]{BS16}
Yuri Berest and Peter Samuelson, \emph{Double affine {H}ecke algebras and
  generalized {J}ones polynomials}, Compos. Math. \textbf{152} (2016), no.~7,
  1333--1384. \MR{3530443}

\bibitem[Bul97]{Bul97}
Doug Bullock, \emph{Rings of {${\rm SL}\sb 2({\bf C})$}-characters and the
  {K}auffman bracket skein module}, Comment. Math. Helv. \textbf{72} (1997),
  no.~4, 521--542. \MR{1600138 (98k:57008)}

\bibitem[BZ03]{BZ03}
Gerhard Burde and Heiner Zieschang, \emph{Knots}, second ed., de Gruyter
  Studies in Mathematics, vol.~5, Walter de Gruyter \& Co., Berlin, 2003.
  \MR{1959408 (2003m:57005)}

\bibitem[CCG{\etalchar{+}}94]{CCG94}
D.~Cooper, M.~Culler, H.~Gillet, D.~D. Long, and P.~B. Shalen, \emph{Plane
  curves associated to character varieties of {$3$}-manifolds}, Invent. Math.
  \textbf{118} (1994), no.~1, 47--84. \MR{1288467 (95g:57029)}

\bibitem[CL98]{CL98}
D.~Cooper and D.~D. Long, \emph{Representation theory and the {$A$}-polynomial
  of a knot}, Chaos Solitons Fractals \textbf{9} (1998), no.~4-5, 749--763,
  Knot theory and its applications. \MR{1628754 (99c:57013)}

\bibitem[DG04]{DG04}
Nathan~M. Dunfield and Stavros Garoufalidis, \emph{Non-triviality of the
  {$A$}-polynomial for knots in {$S\sp 3$}}, Algebr. Geom. Topol. \textbf{4}
  (2004), 1145--1153 (electronic). \MR{2113900}

\bibitem[FG00]{FG00}
Charles Frohman and R{\u{a}}zvan Gelca, \emph{Skein modules and the
  noncommutative torus}, Trans. Amer. Math. Soc. \textbf{352} (2000), no.~10,
  4877--4888. \MR{MR1675190 (2001b:57014)}

\bibitem[FK65]{FK65}
Robert Fricke and Felix Klein, \emph{Vorlesungen \"uber die {T}heorie der
  automorphen {F}unktionen. {B}and 1: {D}ie gruppentheoretischen {G}rundlagen.
  {B}and {II}: {D}ie funktionentheoretischen {A}usf\"uhrungen und die
  {A}ndwendungen}, Bibliotheca Mathematica Teubneriana, B\"ande 3, vol.~4,
  Johnson Reprint Corp., New York; B. G. Teubner Verlagsgesellschaft, Stuttg
  art, 1965. \MR{0183872}

\bibitem[Gel02]{Gel02}
R{\u{a}}zvan Gelca, \emph{Non-commutative trigonometry and the {$A$}-polynomial
  of the trefoil knot}, Math. Proc. Cambridge Philos. Soc. \textbf{133} (2002),
  no.~2, 311--323. \MR{1912404 (2004c:57021)}

\bibitem[Gol97]{Gol97}
William~M. Goldman, \emph{Ergodic theory on moduli spaces}, Ann. of Math. (2)
  \textbf{146} (1997), no.~3, 475--507. \MR{1491446 (99a:58024)}

\bibitem[GS]{M2}
Daniel~R. Grayson and Michael~E. Stillman, \emph{Macaulay2, a software system
  for research in algebraic geometry}, Available at
  \texttt{www.math.uiuc.edu/Macaulay2/}.

\bibitem[GS04]{GS04}
R{\u{a}}zvan Gelca and Jeremy Sain, \emph{The computation of the
  non-commutative generalization of the {$A$}-polynomial of the figure-eight
  knot}, J. Knot Theory Ramifications \textbf{13} (2004), no.~6, 785--808.
  \MR{2088746 (2005f:57020)}

\bibitem[IIS06]{IIS06}
Michi-aki Inaba, Katsunori Iwasaki, and Masa-Hiko Saito, \emph{Dynamics of the
  sixth {P}ainlev\'e equation}, Th\'eories asymptotiques et \'equations de
  {P}ainlev\'e, S\'emin. Congr., vol.~14, Soc. Math. France, Paris, 2006,
  pp.~103--167. \MR{2353464}

\bibitem[IT10]{IT10}
Tatsuro Ito and Paul Terwilliger, \emph{Double affine {H}ecke algebras of rank
  1 and the {$\Bbb Z\sb 3$}-symmetric {A}skey-{W}ilson relations}, SIGMA
  Symmetry Integrability Geom. Methods Appl. \textbf{6} (2010), Paper 065, 9.
  \MR{2725018}

\bibitem[Iwa03]{Iwa03}
Katsunori Iwasaki, \emph{An area-preserving action of the modular group on
  cubic surfaces and the {P}ainlev\'e {VI} equation}, Comm. Math. Phys.
  \textbf{242} (2003), no.~1-2, 185--219. \MR{2018272}

\bibitem[KM04]{KM04}
P.~B. Kronheimer and T.~S. Mrowka, \emph{Dehn surgery, the fundamental group
  and {SU{$(2)$}}}, Math. Res. Lett. \textbf{11} (2004), no.~5-6, 741--754.
  \MR{2106239}

\bibitem[Koo08]{Koo08}
Tom~H. Koornwinder, \emph{Zhedanov's algebra {$\rm AW(3)$} and the double
  affine {H}ecke algebra in the rank one case. {II}. {T}he spherical
  subalgebra}, SIGMA Symmetry Integrability Geom. Methods Appl. \textbf{4}
  (2008), Paper 052, 17. \MR{2425640 (2010e:33028)}

\bibitem[LT11]{LT11}
T.~T.~Q. {Le} and A.~T. {Tran}, \emph{{On the AJ conjecture for knots}},
  arXiv:1111.5258 (2011).

\bibitem[Mag80]{Mag80}
Wilhelm Magnus, \emph{Rings of {F}ricke characters and automorphism groups of
  free groups}, Math. Z. \textbf{170} (1980), no.~1, 91--103. \MR{558891}

\bibitem[MR01]{MR01}
J.~C. McConnell and J.~C. Robson, \emph{Noncommutative {N}oetherian rings},
  revised ed., Graduate Studies in Mathematics, vol.~30, American Mathematical
  Society, Providence, RI, 2001, With the cooperation of L. W. Small.
  \MR{1811901 (2001i:16039)}

\bibitem[Nak13]{Nak13}
Yasuharu Nakae, \emph{A good presentation of {$(-2,3,2s+1)$}-type pretzel knot
  group and {$\mathbb R$}-covered foliation}, J. Knot Theory Ramifications
  \textbf{22} (2013), no.~1, 1250143, 23. \MR{3024026}

\bibitem[NS04]{NS04}
Masatoshi Noumi and Jasper~V. Stokman, \emph{Askey-{W}ilson polynomials: an
  affine {H}ecke algebra approach}, Laredo {L}ectures on {O}rthogonal
  {P}olynomials and {S}pecial {F}unctions, Adv. Theory Spec. Funct. Orthogonal
  Polynomials, Nova Sci. Publ., Hauppauge, NY, 2004, pp.~111--144. \MR{2085854
  (2005h:42057)}

\bibitem[Obl04]{Obl04}
Alexei Oblomkov, \emph{Double affine {H}ecke algebras of rank 1 and affine
  cubic surfaces}, Int. Math. Res. Not. (2004), no.~18, 877--912. \MR{2037756
  (2005j:20005)}

\bibitem[PS00]{PS00}
J{\'o}zef~H. Przytycki and Adam~S. Sikora, \emph{On skein algebras and {${\rm
  Sl}\sb 2({\bf C})$}-character varieties}, Topology \textbf{39} (2000), no.~1,
  115--148. \MR{1710996 (2000g:57026)}

\bibitem[Sah99]{Sah99}
Siddhartha Sahi, \emph{Nonsymmetric {K}oornwinder polynomials and duality},
  Ann. of Math. (2) \textbf{150} (1999), no.~1, 267--282. \MR{1715325
  (2002b:33018)}

\bibitem[Ter11]{Ter11}
Paul Terwilliger, \emph{The universal {A}skey-{W}ilson algebra}, SIGMA Symmetry
  Integrability Geom. Methods Appl. \textbf{7} (2011), Paper 069, 24.
  \MR{2861207}

\bibitem[Ter13]{Ter13}
\bysame, \emph{The universal {A}skey-{W}ilson algebra and {DAHA} of type
  {$(C\sp \vee\sb 1,C\sb 1)$}}, SIGMA Symmetry Integrability Geom. Methods
  Appl. \textbf{9} (2013), Paper 047, 40. \MR{3116183}

\bibitem[Vog89]{Vog89}
H.~Vogt, \emph{Sur les invariants fondamentaux des \'equations
  diff\'erentielles lin\'eaires du second ordre}, Ann. Sci. \'Ecole Norm. Sup.
  (3) \textbf{6} (1889), 3--71. \MR{1508833}

\end{thebibliography}
\bibliographystyle{amsalpha}

\end{document}